\documentclass[11pt]{article}

\RequirePackage{amsthm,amsmath,amssymb,amsfonts,graphicx,bm,bbm,mathrsfs}
\usepackage{authblk}
\usepackage{enumerate}
\usepackage[margin = 1in]{geometry}
\usepackage{parskip}
\usepackage{fancyhdr}
\usepackage{float}
\usepackage{comment}
\usepackage[usestackEOL]{stackengine}
\usepackage{rotating}

\usepackage{fullpage}
\usepackage[OT1]{fontenc}
\usepackage[utf8]{inputenc}
\RequirePackage[colorlinks,citecolor=blue,urlcolor=blue,breaklinks=true]{hyperref}
\usepackage{breakcites}
\usepackage{cleveref, autonum}
\usepackage{booktabs} 
\usepackage{amsthm,amsmath,amssymb,amsfonts,graphicx,bm}
\usepackage{tikz}
\usepackage[round]{natbib}

\usepackage[position=top]{subfig}

\newcommand{\Var}{\mathrm{Var}}

\newcommand{\GP}{\mathrm{GP}}

\newcommand{\bI}{\bm{I}}

\newcommand\bx{\bm{x}}

\newcommand\data{\mathbb{D}_n}

\theoremstyle{plain}
\newtheorem{thm}{Theorem}[section]
\newtheorem{lem}[thm]{Lemma}
\newtheorem{prop}[thm]{Proposition}
\newtheorem{cor}[thm]{Corollary}     

\theoremstyle{remark}     \newtheorem*{rem}{Remark}    
\theoremstyle{Theorem} 

\newtheorem{assumption}{Condition}
\renewcommand{\theassumption}{(A\arabic{assumption})} 

\usepackage[utf8]{inputenc} 
\usepackage[T1]{fontenc}    
\usepackage{hyperref}       
\usepackage{url}            
\usepackage{booktabs}       
\usepackage{amsfonts}       
\usepackage{nicefrac}       
\usepackage{microtype}      

\newcommand{\PP}{\mathbb{P}}
\newcommand{\EE}{\mathbb{E}}
\newcommand{\bbR}{\mathbb{R}}
\newcommand{\bbH}{\mathbb{H}}
\newcommand{\mX}{\mathcal{X}}
\newcommand{\tr}{\mathrm{tr}}

\newcommand{\s}{\sum_{i=1}^{\infty}}
\newcommand{\Ltwo}{L^2_{p_X}(\mX)}

\begin{document}

\title{Equivalence of Convergence Rates of Posterior Distributions and Bayes Estimators for Functions and Nonparametric Functionals}
\author[ ]{Zejian Liu\thanks{zejian.liu@rice.edu}}
\author[ ]{Meng Li \thanks{meng@rice.edu}}
\affil[ ]{Department of Statistics, Rice University}
\date{}

\maketitle

\begin{abstract}
We study the posterior contraction rates of a Bayesian method with Gaussian process priors in nonparametric regression and its plug-in property for differential operators. For a general class of kernels, we establish convergence rates of the posterior measure of the regression function and its derivatives, which are both minimax optimal up to a logarithmic factor for functions in certain classes. Our calculation shows that the rate-optimal estimation of the regression function and its derivatives share the same choice of hyperparameter, indicating that the Bayes procedure remarkably adapts to the order of derivatives and enjoys a generalized plug-in property that extends real-valued functionals to function-valued functionals. This leads to a practically simple method for estimating the regression function and its derivatives, whose finite sample performance is assessed using simulations.

Our proof shows that, under certain conditions, to any convergence rate of Bayes estimators there corresponds the same convergence rate of the posterior distributions (i.e., posterior contraction rate), and vice versa. This equivalence holds for a general class of Gaussian processes and covers the regression function and its derivative functionals, under both the $L_2$ and $L_{\infty}$ norms. In addition to connecting these two fundamental large sample properties in Bayesian and non-Bayesian regimes, such equivalence enables a new routine to establish posterior contraction rates by calculating convergence rates of nonparametric point estimators. 

At the core of our argument is an operator-theoretic framework for kernel ridge regression and equivalent kernel techniques. We derive a range of sharp non-asymptotic bounds that are pivotal in establishing convergence rates of nonparametric point estimators and the equivalence theory, which may be of independent interest. 

\end{abstract}

\section{Introduction}

\sloppy Posterior contraction rates, or convergence rates of the posterior measure, have been widely used in the Bayesian literature to study the asymptotic behavior of posterior distributions. Let $(\mathfrak{X}^{(n)},\mathscr{A}^{(n)},P_\theta^{(n)}:\theta\in\Theta)$ be a sequence of statistical experiments with a possibly infinite-dimensional parameter space $\Theta$ and observations $\mathbb{X}^{(n)}$ indexed by the sample size $n$. If a prior distribution $\Pi$ is put on $\Theta$, a contraction rate of the posterior distribution $\Pi_n(\cdot \mid \mathbb{X}^{(n)})$ at the parameter $\theta_0$ with respect to a semimetric $d$ is a sequence $\epsilon_n$ such that for every $M_n\rightarrow\infty$,
\begin{equation} \label{eq:contraction.def} 
\Pi_n(\theta: d(\theta,\theta_0)\geq M_n\epsilon_n\mid\mathbb{X}^{(n)}) \rightarrow 0
\end{equation}
in $P_{\theta_0}^{(n)}$-probability. General theory for posterior contraction rates in nonparametric Bayes was developed in the seminal work \cite{ghosal2000convergence} by verifying the prior mass and entropy conditions; see also \cite{shen2001rates} for a similar result. The past two decades have seen a surge of interest in establishing posterior contraction rates under a variety of models and prior distributions \citep{ghosal2001entropies,kleijn2006misspecification,ghosal2007convergence,ghosal2007posterior,van2008rates,kruijer2008posterior,castillo2008lower,rousseau2010rates,shen2013adaptive,castillo2014bayesian,van2017adaptive,bhattacharya2019bayesian}. We refer to Section 3.3 in \cite{rousseau2016frequentist} for an excellent review.

In strike contrast to contraction rates that pertain to the entire posterior distribution and are unique to Bayesian inference, we say $\epsilon_n$ is a convergence rate of a point estimator $\hat{\theta}_n$ if there holds
\begin{equation} \label{eq:convergence.def} 
d(\hat{\theta}_n,\theta_0)=O_P(\epsilon_n)
\end{equation}
in $P_{\theta_0}^{(n)}$-probability. Deriving the rate of convergence of $\hat{\theta}_n$ in infinite-dimensional settings to study asymptotic properties has attracted numerous attention and arguably remained a mainstay in nonparametric statistics and learning theory. 

While both convergence rates are concerned with large sample properties of a statistical procedure, they are two distinct subjects. It is widely perceived that posterior contraction rates appear to be a stronger result than convergence rates when $\hat{\theta}_n$ is a Bayes estimator, a point estimator that minimizes the posterior expected loss. Indeed, convergence of the posterior at a rate ensures the existence of point estimators that converge to the true $\theta_0$ at the same rate. For example, if $\epsilon_n$ is a contraction rate, the center $\hat{\theta}_n$ of the smallest ball that contains posterior mass at least $1/2$ may achieve the same convergence rate (cf. Theorem 2.5 in \cite{ghosal2000convergence}). In addition, if we consider a bounded metric $d$ whose square is convex, the posterior mean $\hat{\theta}_n:=\int \theta d\Pi_n(\theta\mid X^{(n)})$, which is a special case of Bayes estimators under the squared error loss \citep{le2012asymptotic}, satisfies that  
\begin{equation}\label{eq:jensen}
d(\hat{\theta}_n,\theta_0)\leq M_n\epsilon_n+\|d\|_\infty^{1/2}\Pi_n(\theta:d(\theta,\theta_0)>M_n\epsilon_n\mid \mathbb{X}^{(n)}), 
\end{equation}
according to a direct application of Jensen's inequality (cf. Theorem~8.8 in \cite{ghosal2017fundamentals} and the discussion therein), where $\|d\|_\infty$ is a bound on the maximal distance. These results indicate that the posterior may not contract at a rate faster than the optimal convergence rate of point estimators. However, little is known in the literature about the converse statement in the nonparametric setting, that is, whether the seemingly simpler convergence rate of Bayes estimators can lead to the rate of convergence of the entire posterior distribution.

We study this problem by considering the nonparametric regression model 
\begin{equation}\label{eq:model}
Y_i=f(X_i)+\varepsilon_i,\quad \varepsilon_i\sim N(0,\sigma^2),
\end{equation}
where the data $\data=\{X_i,Y_i\}_{i=1}^{n}$ are i.i.d. samples from a distribution $\PP_0$ on $\mX\times \bbR$ that is determined by $\PP_X$, $f_0$, and $\sigma^2$, which are respectively the marginal distribution of $X_i$, the true regression function, and the noise variance that is possibly unknown. Let $p_X$ denote the density of $\PP_X$ with respect to Lebesgue measure $\mu$. Here $\mX\subset \bbR^p$ is a compact metric space for $p\geq 1$.

In this paper, we show that under mild conditions, convergence rates of the posterior mean will automatically yield the same convergence rate of the posterior distribution in nonparametric regression models with random design and Gaussian process (GP) priors. The equivalence of two convergence rates immediately bridges the nonparametric Bayesian literature and learning theory, which implies that an asymptotic rate obtained in one of these problems yields the same rate in the other problem. This interesting equivalence between the two different convergences, coupled with convergence rates of point estimators that can be derived systemically using an operator-theoretic framework, provides a new approach that is often technically simpler to establish contraction rates under both the $L_2$ and $L_{\infty}$ norms. In addition, we show that such equivalence applies to differential operators, enabling the study of asymptotic properties in estimating function derivatives. Our calculation indicates that GP priors enjoy a remarkable property that the posterior distribution remains minimax optimal when convoluted with the differential operator; this generalizes the classical plug-in property on real-valued functionals to function-valued functionals, leading to a concept of \textit{nonparametric plug-in property} and providing practically simple methods to make inference on function derivatives.

This paper makes three main contributions. The first main contribution is that we propose a new framework for studying convergence rates of posterior distributions in nonparametric regression, which covers a general class of Gaussian process priors and both the $L_2$ and $L_{\infty}$ norms. In the existing literature, posterior contraction of nonparametric regression based on GP priors often boils down to the concentration function~\citep{van2008rates} to verify the classical prior mass and entropy conditions, which consists of a small ball probability and a concentration measure of the GP prior with certain kernels; for example, see \cite{van2011information, li2017bayesian} for particular kernels therein. In work along this line, the commonly used metrics for establishing posterior contraction rates are Hellinger and total variation distances. While contraction rates under the $L_2$ norm are relatively well studied, there is a limited development on contraction rates under the $L_\infty$ norm, with a few exceptions including \cite{castillo2014bayesian,yoo2016supremum,yang2017frequentist}.

The proposed equivalence provides an alternative approach to establish contraction rates. In particular, rather than rely on the concentration function, one can resort to the convergence of the posterior mean and some mild conditions regarding the eigendecomposition of the covariance kernel of the GP prior, which are both systematically addressed in a unified non-asymptotic analysis (see our third contribution below) and an equivalent kernel technique. Hence, one may build on the well established literature on convergence rates of Bayes estimators, particularly through kernel ridge regression (KRR) \citep{wahba1990spline,cucker2007learning,mendelson2010regularization}, to conveniently obtain posterior contraction rates under both the $L_2$ and $L_{\infty}$ norms. As an example, under the proposed framework, we establish posterior contraction rates for H\"older smooth functions as well as analytic-type functions, which are minimax optimal up to a logarithmic factor.

As the second contribution, we show that the equivalence theory also covers function-valued functionals of the regression function, which hereafter are referred to as \textit{nonparametric functionals}. As an important implication, we find GP priors enjoy a ``nonparametric plug-in property'' that generalizes the plug-in property coined by \cite{bickel2003nonparametric} from real-valued to nonparametric functionals. A nonparametric estimator attains the plug-in property if it simultaneously achieves the minimax convergence rate for estimating the unknown function and the parametric $1/\sqrt{n}$-rate for estimating some bounded linear functionals. Although the plug-in property was originally proposed in the frequentist context, it holds for many Bayes estimators \citep{castillo2013nonparametric}. However, (bounded) linear functionals considered in the plug-in property is restrictive in that the derivatives of a function at a fixed point are excluded from such functionals. In practice, the entire derivative function as opposed to its evaluation at a fixed point might be more of interest \citep{holsclaw2013gaussian,dai2018derivative}, which to date is a largely unaddressed problem; one notable exception is~\cite{yoo2016supremum}, which studies a random series prior based on tensor product of B-splines. We extend the equivalence theory to the differential operators; in particular, this leads to a new approach to infer function derivatives in nonparametric regression. We show that the posterior distribution concentrates at such nonparametric functional at a nearly minimax rate \citep{stone1982optimal} in specific examples, thus achieving the nonparametric plug-in property. In addition, we find the rate-optimal estimation of $f_0$ and its derivatives share the same choice of hyperparameter in the GP prior, which is a remarkable property as it indicates the Bayes procedure automatically adapts to the order of derivative in this case. Finite sample performance of the proposed nonparametric plug-in procedure is assessed through simulations.

The third main contribution of the present paper is that we provide a non-asymptotic analysis of the convergence of Bayes estimators, which is also closely related to KRR. Our non-asymptotic analysis is built on an operator-theoretic framework \citep{smale2005shannon,smale2007learning} and the equivalent kernel technique, covering the estimation of posterior mean and variance, derivatives of posterior mean and posterior variances of derivatives of Gaussian processes. We substantially extend the analysis in our earlier work~\citep[Technical Report,][]{liu2020non} by deriving a range of either new or sharper error bounds. The developed non-asymptotic bounds construct rate-optimal estimators of the regression function and its derivatives, and are also crucial in proving the equivalence theory and nonparametric plug-in property. We remark that the our non-asymptotic analysis may be of independent interest in broader contexts such as function estimation and information theory.

When the error variance $\sigma^2$ is unknown, we adopt an empirical Bayes scheme where we estimate the error variance by its marginal maximum likelihood estimator (MMLE). We show that the MMLE is consistent, and all established results on the equivalence of posterior convergence and contraction hold under the empirical Bayes scheme.

\subsection{Notation}
We write $X=(X_1^T,\ldots,X_n^T)^T\in\bbR^{n\times p}$ and $Y=(Y_1,\ldots,Y_n)^T\in \bbR^n$. Let $\|\cdot\|$ be the Euclidean norm; for $f, g: \mX \rightarrow \bbR$, let $\|f\|_\infty$ be the $L_\infty$ (supremum) norm, $\|f\|_2=(\int_\mX f^2d\PP_X)^{1/2}$ the $L_2$ norm with respect to the covariate distribution $\PP_X$, and $\left<f,g\right>_2=(\int_\mX fgd\PP_X)^{1/2}$ the inner product. The corresponding $L_2$ space relative to $\PP_X$ is denoted by $\Ltwo$; we write $L^2(\mX)$ as the $L_2$ space with respect to Lebesgue measure $\mu$. Denote the space of all essentially bounded functions by $L^\infty(\mX)$. Let  $\mathbb{N}$ be the set of all positive integers and write $\mathbb{N}_0=\mathbb{N}\cup\{0\}$. We let $C(\mX)$ and $C(\mX,\mX)$ denote the space of continuous functions and continuous bivariate functions. In one-dimensional case, for $\Omega\subset\bbR$, a function $f:\Omega\rightarrow\bbR$ and $k\in\mathbb{N}$, we use $f^{(k)}$ to denote its $k$th derivative as long as it exists and $f^{(0)}=f$. Let $C^m(\Omega)=\{f:\Omega\rightarrow\bbR \mid f^{(k)}\in C(\Omega) \ {\rm for\ all\ } 1\leq k\leq m\}$ denote the space of $m$-times continuously differentiable functions and $C^{2m}(\Omega,\Omega)=\{K:\Omega\times \Omega\rightarrow\bbR \mid \partial^k_x\partial^k_{x'}K(x,x')\in C(\Omega,\Omega) \ {\rm for\ all\ } 1\leq k\leq m\}$ denote the space of $m$-times continuously differentiable bivariate functions, where $\partial^k_x=\partial^k/\partial x^k$. For two sequences $a_n$ and $b_n$, we write $a_n \lesssim b_n$ if $a_n \leq C b_n$ for a universal constant $C>0$, and $a_n \asymp b_n$ if $a_n \lesssim b_n$ and $b_n \lesssim a_n$.

\subsection{Organization}
In Section~\ref{sec:main.results} we describe our main results on the equivalence of posterior contraction rates and convergence rates when estimating the regression function, based on which we provide an alternative method to derive the nearly minimax optimal contraction rate for estimating functions in H\"older class under the $L_2$ and $L_\infty$ norms using kernels with polynomially decaying eigenvalues. The nonparametric plug-in property is presented in Section~\ref{sec:operator}, where we focus on the differential operator. Nearly minimax contraction rates for estimating the derivatives of H\"older functions are obtained. In Section~\ref{sec:MMLE} we propose an empirical Bayes scheme to address unknown error variance and establish the same equivalence theory.  Section~\ref{sec:non-asymptotic} is devoted to the non-asymptotic analysis based on an operator-theoretic approach and equivalent kernel technique, where we study the convergence rates of a range of quantities, including the posterior mean and variance, derivatives of the posterior mean,  and posterior variances of derivatives of GPs. Section~\ref{sec:simulation} carries out a simulation study to assess the finite sample performance of the proposed method for estimating the regression function and its derivative. All proofs are collected in Section~\ref{sec:proof}. 

\section{Equivalence of contraction and convergence rates}\label{sec:main.results}

Throughout the paper, we assume the true regression function $f_0\in\Ltwo$. We start with introducing the GP prior, posterior conjugacy, and the eigendecomposition of the covariance kernel.

\subsection{Preliminary: Prior and posterior conjugacy}
We assign a Gaussian process prior $\Pi$ on the regression function $f \sim \GP(0, \sigma^2 (n \lambda)^{-1} K)$, where $K(\cdot,\cdot):\mX\times\mX\rightarrow \bbR$ is a continuous, symmetric and positive definite bivariate function (i.e., the \textit{Mercer kernel}), and $\lambda>0$ is a regularization parameter that possibly depends on the sample size $n$. The rescaling factor $(n\lambda)^{-1}$ in the covariance kernel connects Bayes estimator with the kernel ridge regression \citep{wahba1990spline,cucker2007learning}; see also~\cite{yang2017frequentist} and Theorem 11.61 in \cite{ghosal2017fundamentals} for more discussion on this connection. 

By conjugacy, the posterior distribution $\Pi_n(\cdot\mid\data)$ is still a GP: $f |\data \sim \GP(\hat{f}_n, \hat{V}_n)$, where the posterior mean $\hat{f}_n$ and posterior covariance $\hat{V}_n$ are given by
\begin{align}
\hat{f}_n(\bx) &= K(\bx, X) [K(X, X) + n \lambda \bI_n]^{-1} Y,\\
\label{eq:variance}\hat{V}_n(\bx, \bx') &=  \sigma^2 (n \lambda)^{-1} \{K(\bx, \bx') - K(\bx, X)[K(X, X) + n \lambda \bI_n]^{-1} K(X, \bx')\},
\end{align}
for any $\bx, \bx' \in \mX$. Here $K(X, X)$ is the $n$ by $n$ matrix $(K(X_i, X_j))_{i, j = 1}^n$ and $K(\bx, X)$ is the 1 by $n$ vector $(K(\bx, X_i))_{i = 1}^n$. We write $\hat{V}_n(\bx):=\hat{V}_n(\bx,\bx)$ for the marginal posterior variance at $\bx$. 

The property of a GP prior is largely determined by its covariance kernel $K$.
By Mercer's Theorem, there exists an orthonormal basis $\{\phi_i\}_{i=1}^\infty$ of $\Ltwo$ and $\{\mu_i\}_{i=1}^\infty$ with $\mu_1\geq\mu_2\geq \cdots> 0$ and $\lim\limits_{i\rightarrow\infty}\mu_i=0$ such that for any $\bx,\bx'\in\mX$,
\begin{equation}
\label{eq:eigendecomposition}
K(\bx,\bx')=\sum_{i=1}^{\infty}\mu_i\phi_i(\bx)\phi_i(\bx'),
\end{equation}
where the convergence is absolute and uniform. We call $\{\mu_i\}_{i=1}^\infty$ and $\{\phi_i\}_{i=1}^\infty$ the eigenvalues and the eigenfunctions of $K$. Then the reproducing kernel Hilbert space (RKHS) $\bbH$ induced by $K$ can be characterized by a series representation 
\begin{equation}\label{eq:RKHS.norm}
\bbH=\left\{f\in \Ltwo: \|f\|^2_{\bbH}=\sum_{i=1}^{\infty}\frac{f_i^2}{\mu_i}<\infty, f_i=\left<f,\phi_i\right>_{2}\right\},
\end{equation}
equipped with the inner product $\left<f,g\right>_{\bbH}=\sum_{i=1}^{\infty}{f_ig_i}/{\mu_i}$ for any $f=\sum_{i=1}^{\infty}f_i\phi_i$ and $g=\sum_{i=1}^{\infty}g_i\phi_i$ in $\bbH$.

It is well known that the posterior mean $\hat{f}_n$ under the considered GP prior coincides with the kernel ridge regression estimator, which is defined by the following optimization problem:
\begin{equation}\label{eq:KRR} 
\hat{f}_n = \underset{f \in \bbH}{\arg\min} \left\{\frac{1}{n}\sum_{i = 1}^n (Y_i - f(X_i))^2 + \lambda \|f\|^2_{\bbH}\right\}.
\end{equation}

Our proofs also make extensive use of the so-called equivalent kernel $\tilde{K}$ \cite[Chapter~7]{rasmussen2006}, which shares the same eigenfunctions with $K$ with altered eigenvalues $\nu_i={\mu_i}/{(\lambda+\mu_i)}$ for $i\in\mathbb{N}$, i.e., 
\begin{equation}\label{eq:equivalent.kernel}
\tilde{K}(\bx,\bx')=\s \nu_i\phi_i(\bx)\phi_i(\bx').
\end{equation}
Note that $\tilde{K}$ is also a Mercer kernel. Let $\tilde{\bbH}$ be the RKHS induced by $\tilde{K}$, which is equivalent to $\bbH$ as a function space, but with a different inner product
\begin{equation}
\left<f,g\right>_{\tilde{\bbH}}=\left<f,g\right>_2+\lambda\left<f,g\right>_\bbH.
\end{equation}
We call the corresponding norm $\|\cdot\|_{\tilde{\bbH}}$ the equivalent RKHS norm. The equivalent RKHS norm upper bounds the $L_\infty$ norm by noting that
\begin{equation}
f(\bx)=\left<f,\tilde{K}_{\bx}\right>_{\tilde{\bbH}}\leq \|f\|_{\tilde{\bbH}} \|K_{\bx}\|_{\tilde{\bbH}}= \sqrt{\tilde{K}(\bx,\bx)}\|f\|_{\tilde{\bbH}},
\end{equation}
which follows from Cauchy-Schwarz inequality. Taking the supremum on both sides, we obtain that for any $f\in\tilde{\bbH}$, it holds
\begin{equation}\label{eq:sup.RKHS.norm}
\|f\|_\infty\leq \tilde{\kappa} \|f\|_{\tilde{\bbH}},
\end{equation}
where we define $\tilde{\kappa}^2:=\sup_{\bx\in\mX}\tilde{K}(\bx,\bx)$.

\subsection{Main results on the equivalence}
We assume the following conditions on the eigenfunctions of the covariance kernel.

\begin{assumption}
The eigenfunctions of $K$ are uniformly bounded, i.e., there exists a constant $C_\phi>0$ such that $\|\phi_i\|_\infty\leq C_\phi$ for all $i\in\mathbb{N}$.
\end{assumption}

\begin{assumption}
For any $\bx,\bx'\in\mX$, there exists $L_\phi>0$ such that $|\phi_i(\bx)-\phi_i(\bx')| \leq L_\phi i \|\bx-\bx'\|$ for any $i\in\mathbb{N}$.
\end{assumption}

Recall that the equivalent kernel $\tilde{K}$ assumes an eigendecompisition with the same eigenfunctions as those of $K$. Under Condition (A1), we have
\begin{equation}\label{eq:def.kappa}
\tilde{\kappa}^2\leq C_\phi^2\sum_{i=1}^{\infty}\nu_i\lesssim\sum_{i=1}^{\infty}\frac{\mu_i}{\lambda+\mu_i},
\end{equation}
where the last expression is the so-called \textit{effective dimension} of the kernel $K$ with respect to $p_X$~\citep{zhang2005learning}. We also define a high-order counterpart
\begin{equation}\label{eq:tilde01}
\hat{\kappa}_{01}^2:=\s\frac{i\mu_i}{\lambda+\mu_i}.
\end{equation}

We now present our first equivalence result concerning the estimation of the regression function.

\begin{thm}\label{thm:contraction}
Let $\epsilon_n$ be a sequence such that $\epsilon_n\rightarrow0$ and $n\epsilon_n^2\rightarrow\infty$. Suppose $\lambda$ is chosen such that $\tilde{\kappa}^2=o(\sqrt{n/\log n})$, $\tilde{\kappa}^2=O(n\epsilon_n^2/\log n)$ and $\hat{\kappa}_{01}^2=o(n)$. Under Conditions (A1) and (A2), $\epsilon_n$ is a convergence rate of $\hat{f}_n$ in $\PP_0^{(n)}$-probability if and only if it is a posterior contraction rate of $\Pi_n(\cdot \mid \data)$ at $f_0$; this equivalence holds for any bounded $f_0 \in \Ltwo$ under the $\|\cdot\|_p$ norm for $p = 2$ or $\infty$. 
\end{thm}

Using the triangle inequality, the posterior mean $\hat{f}_n$ in Theorem~\ref{thm:contraction} can be replaced by any efficient Bayes estimator $\tilde{f}_n$ such that $\tilde{f}_n=\hat{f}_n+O_P(\epsilon_n)$ in $\PP_0^{(n)}$-probability. We formalize this result in the following corollary.
\begin{cor}\label{cor:contraction}

Under the conditions of Theorem~\ref{thm:contraction}, for any estimator $\tilde{f}_n$ such that $\|\tilde{f}_n - \hat{f}_n\|_p = O_P(\epsilon_n)$ in $\PP_0^{(n)}$-probability,
$\epsilon_n$ is a convergence rate of $\tilde{f}_n$ if and only if it is a posterior contraction rate of $\Pi_n(\cdot \mid \data)$ at $f_0$; this equivalence holds for any bounded $f_0 \in \Ltwo$ under the $\|\cdot\|_p$ norm for $p = 2$ or $\infty$.

\end{cor}

\begin{rem}
The conditions of Theorem~\ref{thm:contraction} can be verified through direct calculations based on the decay rate of eigenvalues. For example, we provide the rates of $\tilde{\kappa}^2$ and $\hat{\kappa}_{01}^2$ for kernels with polynomially decaying eigenvalues later in the paper in Lemma~\ref{lem:differentiability.matern}.
\end{rem}

\begin{rem}
Theorem~\ref{thm:contraction} provides a new approach for establishing posterior contraction rates. Thanks to the established equivalence, computing the contraction rate now boils down to analyzing the convergence rate of the posterior mean as well as the eigendecomposition of covariance kernel, which are well studied problems in nonparametric statistics. Moreover, in view of the connection between posterior mean and kernel ridge regression estimator, we can also take advantage of the rich literature on KRR and kernel learning theory. Along this line, in Section~\ref{sec:Matern} we will derive the minimax optimal contraction rate for estimating H\"older smooth functions and analytic-type functions under the $L_2$ norm by using kernels with polynomially and exponentially decaying eigenvalues, respectively. As an implication of the equivalence theory, if a certain contraction rate is desired, one may start with constructing a Bayes point estimator that concentrates at the same rate, and then verify the conditions based on the eigendecomposition of covariance kernel.
\end{rem}

\subsection{Application to kernels with polynomially decaying eigenvalues}\label{sec:Matern}
Now we present a concrete example as a direct application of Theorem~\ref{thm:contraction}, where we use kernels with polynomially decaying eigenvalues $K_\alpha$ for the covariance kernel in the GP prior.

Specifically, we consider $\mX=[0,1]$ along with a uniform sampling process for $p_X$, where the corresponding probability measure $\PP_X$ becomes the Lebesgue measure $\mu$. We also assume that $K_\alpha$ satisfies an eigendecomposition with respect to $\mu$ with polynomially decaying eigenvalues, that is, 
\begin{equation}
\mu_i\asymp i^{-2\alpha},\quad i\in\mathbb{N}
\end{equation}
for some $\alpha>0$. This assumption is also made in \cite{amini2012sampled}, \cite{zhang2015divide} and \cite{yang2017frequentist}. We also assume that the eigenfunctions of $K_\alpha$ are the Fourier basis functions
\begin{equation}\label{eq:fourier}
\psi_1(x)=1,\ \psi_{2i}(x)=\sqrt{2}\cos(2\pi ix),\ \psi_{2i+1}=\sqrt{2}\sin(2\pi ix),\quad i\in\mathbb{N},
\end{equation}
which clearly satisfies Conditions (A1) and (A2) with $C_\phi=\sqrt{2}$ and $L_\phi=2\sqrt{2}\pi$.

We denote the equivalent kernel of $K_\alpha$ by $\tilde{K}_\alpha$, and the RKHS induced by $K_\alpha$ and $\tilde{K}_\alpha$ by $\bbH_\alpha$ and $\tilde{\bbH}_\alpha$, respectively. We consider the true regression function $f_0$ to lie in the H\"older space $H^{\alpha}[0,1]$:
\begin{equation}
H^{\alpha}[0,1]=\left\{f\in L^2[0,1]: \|f\|^2_{H^{\alpha}[0,1]}=\sum_{i=1}^{\infty}i^{\alpha}|f_i|<\infty,f_i=\left<f,\psi_i\right>_2\right\}.
\end{equation}
For any $f\in H^{\alpha}[0,1]$, $f$ has continuous derivatives up to order $\lfloor\alpha\rfloor$ and the $\lfloor\alpha\rfloor$th derivative is Lipschitz continuous of order $\alpha-\lfloor\alpha\rfloor$. 

Using our operator-theoretic approach in Section~\ref{sec:non-asymptotic}, it can be first shown that under the choice of regularization parameter $\lambda\asymp ({\log n}/{n})^{\frac{2\alpha}{2\alpha+1}}$, the posterior mean converges to the ground truth such that
\begin{equation}
\|\hat{f}_n-f_0\|_2 \lesssim \left(\frac{\log n}{n}\right)^{\frac{\alpha}{2\alpha+1}}
\end{equation}
in $\PP_0^{(n)}$-probability, where the rate is nearly minimax optimal under the $L_2$ norm (cf. Theorem 2.8 in \cite{tsybakov2008introduction}). We defer the formal statement to Lemma~\ref{lem:minimax.matern}. Therefore, by verifying the regularity conditions in Theorem~\ref{thm:contraction}, we conclude that with the same choice of $\lambda$, the posterior distribution will contract at $f_0$ at the same nearly minimax optimal rate under the $L_2$ norm.

\begin{thm}\label{thm:matern.contraction}
Suppose $f_0\in H^\alpha[0,1]$ for $\alpha>1/2$. If $K_\alpha$ is used in the GP prior with the regularization parameter $\lambda\asymp ({\log n}/{n})^{\frac{2\alpha}{2\alpha+1}}$, then the posterior distribution $\Pi_n(\cdot\mid\data)$ contracts at $f_0$ at the nearly minimax optimal rate $\epsilon_n = \left(\log n/n\right)^{\frac{\alpha}{2\alpha+1}}$ under the $L_2$ norm.
\end{thm}

\begin{rem}
Similar results are provided in Theorem 5 of \cite{van2011information}, but the authors did not consider pointwise convergence of posterior mean and used concentration function as the proof technique instead. Indeed, our equivalence theory allows one to derive posterior contraction under the $L_\infty$ norm based on $L_\infty$ convergence rates of $\hat{f}_n$ in the KRR literature. For example, under a very similar setting, Corollary 2.1 in \cite{yang2017frequentist} shows that
\begin{equation}
\|\hat{f}_n-f_0\|_\infty\lesssim \left(\frac{\log n}{n}\right)^{\frac{\alpha}{2\alpha+1}}
\end{equation}
in $\PP_0^{(n)}$-probability with the same choice of regularization parameter $\lambda\asymp ({\log n}/{n})^{\frac{2\alpha}{2\alpha+1}}$. Applying Theorem~\ref{thm:contraction} with $p=\infty$, we obtain the posterior contraction rate at $f_0$ to be $\epsilon_n = \left(\log n/n\right)^{\frac{\alpha}{2\alpha+1}}$.
\end{rem}

\subsection{Application to kernels with exponentially decaying eigenvalues}
In this section, we consider kernels with exponentially decaying eigenvalues on the unit support $\mX=[0,1]$. Suppose the covariance kernel has an eigendecomposition relative to $\PP_X$ such that the eigenvalues satisfy
\begin{equation}\label{eq:exponential.kernel}
\mu_i\asymp e^{-2\gamma i},\quad i\in\mathbb{N},
\end{equation}
for some $\gamma>0$, and the eigenfunctions satisfy Conditions (A1) and (A2). We denote such kernels by $K_\gamma$. The well-known squared exponential kernel can be approximately viewed as one example of $K_\gamma$, which assumes a closed-form eigendecomposition with respect to a Gaussian sampling process on the real line \citep{rasmussen2006,pati2015adaptive}.

We assume $f_0$ belongs to the analytic-type function class $A^\gamma[0,1]$:
\begin{equation}
A^\gamma[0,1]=\left\{f\in L^2_{p_X}[0,1]: \|f\|^2_{A^\gamma[0,1]}= \sum_{i=1}^{\infty}e^{\gamma i}|f_i|<\infty,f_i=\left<f,\phi_i\right>_2\right\}.
\end{equation}

It will be shown in Lemma~\ref{lem:minimax.exp} that under the choice of regularization parameter $\lambda\asymp 1/n$, the posterior mean based on GP priors with $K_\gamma$ converges to the ground truth at the nearly parametric rate under the $L_2$ norm, i.e.,
\begin{equation}
\|\hat{f}_n-f_0\|_2 \lesssim \frac{\log n}{\sqrt{n}}
\end{equation}
in $\PP_0^{(n)}$-probability. Invoking our equivalence theory, we immediately conclude the same nearly parametric $L_2$ contraction rate for estimating analytic-type functions.

\begin{thm}\label{thm:exponential.contraction}
Suppose $f_0\in A^\gamma[0,1]$ for $\gamma>0$. If $K_\gamma$ is used in the GP prior with the regularization parameter $\lambda\asymp 1/n$, then the posterior distribution $\Pi_n(\cdot\mid\data)$ contracts at $f_0$ at the nearly parametric rate $\epsilon_n=\log n/\sqrt{n}$ under the $L_2$ norm.
\end{thm}

\section{Equivalence theory for differential operators and nonparametric plug-in property}\label{sec:operator}

In this section, we extend the equivalence of posterior contraction rates and convergence rates of point estimators to function derivatives. As a key implication of this equivalence theory, we establish posterior contraction rates for function derivatives and show that Gaussian process priors enjoy a remarkable \textit{nonparametric plug-in property}. We focus on univariate functions in this section, but our argument can be extended to multivariate functions with mixed-partial derivatives. 

The ``plug-in property'' proposed in \cite{bickel2003nonparametric} refers to the phenomenon that a rate-optimal nonparametric estimator also efficiently estimates some bounded linear functionals. A parallel concept has been studied in the Bayesian paradigm relying on posterior distributions and posterior contraction rates~\citep{rivoirard2012bernstein, castillo2013nonparametric, castillo2015bernstein}. 

We first remark that function derivatives may not fall into the classical plug-in property framework. To see this, let $D_t=f'(t)$ be a functional which maps $f$ to its derivative at any fixed point $t\in[0,1]$. It is easy to see that $D_t$ is a linear functional. However, the following Proposition~\ref{prop:differentiation}~\citep[page 13]{conway1994course} shows that $D_t$ is not bounded.
\begin{prop}\label{prop:differentiation}
Let $t\in[0,1]$ and define $D_t:C^1[0,1]\rightarrow \bbR$ by $D_t(f)=f'(t)$. Then, there is no bounded linear functional on $L^2[0,1]$ that agrees with $D_t$ on $C^1[0,1]$.
\end{prop}
Therefore, it appears difficult to analyze function derivatives evaluated at a fixed point, as existing work on the plug-in property typically assumes the functional to be bounded~\citep{bickel2003nonparametric, castillo2013nonparametric, castillo2015bernstein}. In addition, differential operator that maps a function to its derivative functions have a wide range of applications in statistics~\citep{holsclaw2013gaussian,dai2018derivative}; they point to function-valued functionals, or \textit{nonparametric functionals}, as oppose to real-valued functionals that are considered in the classical plug-in property literature. As such, we generalize the plug-in property for real-valued functionals to function-valued functionals, and term it as nonparametric plug-in property.

Underlying the success of establishing the nonparametric plug-in property of GP priors is an extension of the equivalence theory in the preceding section to differential operators. Before presenting the equivalence theory, we first provide a precise definition of nonparametric plug-in procedures for differential operators. Define the $k$-th differential operator $D^k:C^k[0,1]\rightarrow C[0,1]$ by $D^k(f)=f^{(k)}$. Note that if $K\in C^{2k}([0,1],[0,1])$, the posterior distribution of the derivative $f^{(k)}|\data$, denoted by $\Pi_{n,k}(\cdot\mid\data)$, is also a Gaussian process since differentiation is a linear operator. In particular, $f^{(k)} |\data \sim {\rm GP}(\hat{f}^{(k)}_n, \tilde{V}^k_n)$, where
\begin{align}
\label{eq:f.hat.deriv}\hat{f}^{(k)}_n(x) &= K_{k0}(x, X) [K(X, X) + n \lambda \bI_n]^{-1} Y,\\
\label{eq:deriv.variance}\tilde{V}^k_n(x, x')&= \sigma^2 (n \lambda)^{-1} \left\{K_{kk}(x, x') -  K_{k0}(x, X)[K(X, X) + n \lambda \bI_n]^{-1} K_{0k}(X, x')\right\}.
\end{align}
Here $K_{k0}(x, X)$ is the 1 by $n$ vector $(\frac{\partial^k K(x, X_i)}{\partial x^k})_{i = 1}^n$ and $K_{kk}(x, x'):=\frac{\partial^{2k}K(x,x')}{\partial x^k\partial x'^k}$. Then the nonparametric plug-in procedure for $D^k$ refers to the use of the plug-in posterior measure $\Pi_{n,k}(\cdot\mid\data)$, which also induces its point estimator counterpart $\hat{f}^{(k)}_n$ as the plug-in estimator for $f_0^{(k)}$, and the nonparametric plug-in property refers to the optimality of the contraction rate of $\Pi_{n,k}(\cdot\mid\data)$ contracts at $f_0^{(k)}.$

We further make the following assumption on the differentiability of the eigenfunctions.

\renewcommand{\theassumption}{(B)}
\begin{assumption}
The eigenfunctions $\phi_i\in C^{k}[0,1]$ for all $i\in\mathbb{N}$. For any $x,x'\in[0,1]$, there exists $L_{k,\phi}>0$ such that $|\phi^{(k)}_i(x)-\phi^{(k)}_i(x')| \leq L_{k,\phi} i^{k+1} |x-x'|$ for any $i\in\mathbb{N}$.
\end{assumption}

Under Condition (B), we define high-order analogies of effective dimension for any $k\in\mathbb{N}$:
\begin{equation}\label{eq:high.order.kappa}
\tilde{\kappa}_{kk}^2:=\sup_{x,x'\in[0,1]}\partial_{x}^k\partial_{x'}^k\tilde{K}(x,x')=\sup_{x\in[0,1]}\s\frac{\mu_i}{\lambda+\mu_i}\phi^{(k)}_i(x)^2,
\end{equation}
\begin{equation}\label{eq:kappa.k+1}
\hat{\kappa}_{k+1,k+1}^2:=\s\frac{i^{2k+2}\mu_i}{\lambda+\mu_i},
\end{equation}
which depend on the regularization parameter $\lambda$. The following theorem extends the previous equivalence theory to the nonparametric plug-in procedure for the differential operator. 

\begin{thm}\label{thm:deriv.contraction.l2}
Let $\epsilon_n$ be a sequence such that $\epsilon_n\rightarrow0$ and $n\epsilon_n^2\rightarrow\infty$. Suppose that $K\in C^{2k}([0,1],[0,1])$ and $\lambda$ is chosen such that $\tilde{\kappa}^2=o(\sqrt{n/\log n})$, $\tilde{\kappa}_{kk}^2=O(n\epsilon_n^2/\log n)$ and $\tilde{\kappa}_{kk}\hat{\kappa}_{k+1,k+1}=o(n)$. Under Conditions (A1) and (B), $\epsilon_n$ is a convergence rate of $\hat{f}_n^{(k)}$ in $\PP_0^{(n)}$-probability if and only if it is a posterior contraction rate of $\Pi_{n,k}(\cdot \mid \data)$ at $f_0^{(k)}$; this equivalence holds for any $f_0 \in C^k[0,1]$ under the $L_p$ norm for $p=2$ or $\infty$. 
\end{thm}

As an example, we consider $f_0\in H^\alpha[0,1]$ and use $K_\alpha$ for the covariance function along with a uniform sampling process. Our non-asymptotic analysis in Section~\ref{sec:derivatives} will show that the derivatives of posterior mean achieve nearly minimax optimal convergence rates. In particular, Theorem~\ref{thm:minimax.diff.matern} states that for any $k<\alpha-1/2$ and $k\in\mathbb{N}$, it holds
\begin{equation}
\|\hat{f}_n^{(k)}-f_0^{(k)}\|_2\lesssim \left(\frac{\log n}{n}\right)^{\frac{\alpha-k}{2\alpha+1}}
\end{equation}
in $\PP_0^{(n)}$-probability, with the regularization parameter $\lambda\asymp ({\log n}/{n})^{\frac{2\alpha}{2\alpha+1}}$. By verifying the conditions in Theorem~\ref{thm:deriv.contraction.l2}, we obtain the nonparametric plug-in property for the differential operator: the posterior distribution contracts at a nearly minimax optimal rate under the $L_2$ norm when estimating the derivatives of functions in the H\"older class \citep{stone1982optimal}.

\begin{thm}\label{thm:contraction.deriv}
Suppose $f_0\in H^\alpha[0,1]$ for $\alpha>1/2$. If $K_\alpha$ is used in the GP prior with the regularization parameter $\lambda\asymp ({\log n}/{n})^{\frac{2\alpha}{2\alpha+1}}$, then for any $k<\alpha-3/2$ and $k\in\mathbb{N}$, the posterior distribution $\Pi_{n,k}(\cdot\mid\data)$ contracts at $f_0^{(k)}$ at the nearly minimax optimal rate $\epsilon_n = (\log n/n)^{\frac{\alpha-k}{2\alpha+1}}$ under the $L_2$ norm.
\end{thm}

\begin{rem}
Given the smoothness level of the regression function, the rate-optimal estimation of $f_0$ and its derivatives are achieved under the same choice of the regularization parameter $\lambda$ that does not depend on the derivative order, which holds for both convergence of the posterior mean (Theorem~\ref{thm:minimax.diff.matern}) and contraction of the posterior distribution (Theorem~\ref{thm:contraction.deriv}). Therefore, the GP prior enjoys a remarkable property that it automatically adapts to the order of the derivative to be estimated. 
\end{rem}

\section{Unknown error variance}\label{sec:MMLE}

In practice, the error variance $\sigma^2$ in the regression model is unknown and needs to be estimated. Let $\sigma^2_0$ be its true value. \cite{van2009adaptive} proposed a fully Bayesian scheme by endowing the standard error $\sigma$ with a hyperprior, which is supported on a compact interval $[a,b]\subset (0,\infty)$ that contains $\sigma_0$ with a Lebesgue density bounded away from zero. This approach has been followed by many others \citep{bhattacharya2014anisotropic,li2020comparing}. \cite{de2013semiparametric} showed a Bernstein-von Mises theorem for the marginal posterior of $\sigma$, where the prior for $\sigma$ is relaxed to be supported on $(0,\infty)$. In other words, it is possible to simultaneously estimate the regression function at an optimal nonparametric rate and the standard deviation at a parametric rate.

Here we consider an empirical Bayes approach, which is widely used in practice and eliminates the need of jointly sampling $f$ and $\sigma$. Under model \eqref{eq:model}, the marginal likelihood is
\begin{equation}
Y|X\sim N(0, \sigma^2(n\lambda)^{-1}K(X,X)+\sigma^2\bI_n).
\end{equation}
We estimate $\sigma^2$ by its maximum marginal likelihood estimator (MMLE)
\begin{equation}\label{eq:EB.sigma}
\hat{\sigma}^2_n=\lambda Y^T[K(X,X)+n\lambda \bI_n]^{-1}Y.
\end{equation}
Then we endow a new prior $f \sim \GP(0, \hat{\sigma}^2_n (n \lambda)^{-1} K)$ by substituting $\sigma^2$ in the original GP prior with $\hat{\sigma}^2_n$. The induced posterior measure of $f$ is denoted by $\Pi_{n,{\rm EB}}(\cdot\mid\data)$. For the nonparametric plug-in procedure for the $k$-th order derivative, we write $\Pi_{n,k,\text{EB}}(\cdot\mid\data)$ as the posterior distribution.

The next theorem shows that the equivalence theory holds under the empirical Bayes scheme.

\begin{thm}\label{thm:eb.contraction}
If $\lambda$ is chosen such that $n\lambda\rightarrow \infty$, then all results including Theorem~\ref{thm:contraction}, Corollary~\ref{cor:contraction}, Theorem~\ref{thm:matern.contraction}, Theorem~\ref{thm:exponential.contraction}, Theorem~\ref{thm:deriv.contraction.l2}, and Theorem~\ref{thm:contraction.deriv} hold under the empirical Bayes scheme.
\end{thm}
\begin{rem}
The additional condition $n\lambda\rightarrow\infty$ is not restrictive. For instance, it is satisfied in the example specialized to $K_\alpha$ in Section~\ref{sec:Matern} by noting that $\lambda\asymp ({\log n}/{n})^{\frac{2\alpha}{2\alpha+1}}$.
\end{rem}

\section{Non-asymptotic analysis based on equivalent kernel}
\label{sec:non-asymptotic}
In this section, we present a non-asymptotic analysis for an extensive list of key quantities in GP regression, encompassing the posterior mean, the posterior variance, and their derivative counterparts, which is based on an operator-theoretic approach~\citep{smale2005shannon,smale2007learning} and the equivalent kernel technique. Our non-asymptotic analysis applies to general Mercer kernels with uniformly bounded eigenfunctions (Condition (A1)). We substantially extend the analysis in our earlier work~\citep[Technical Report,][]{liu2020non} by deriving a range of either new or sharper error bounds, which are critical to conclude minimax optimality and the equivalence theory for the regression function and differential operators. As a byproduct, our analysis provides a series of convergence rates for point estimators of the regression function and its derivatives, and characterizes uncertainty in GP regression by bounding its posterior variances. This non-asymptotic framework may also be of independent interest.

We begin with introducing operator-theoretic preliminaries and derive error bounds for the posterior mean estimator in Section~\ref{sec:preliminaries}. 
\subsection{Error bounds for Mercer kernels} \label{sec:preliminaries}
For any $f\in\Ltwo$, we define an integral operator $L_K: \Ltwo \rightarrow \bbH$: 
\begin{equation} \label{eq:int.operator}
L_K(f)(\bx) = \int_{\mX} K(\bx, \bx') f(\bx') d\PP_X(\bx'), \quad \bx \in \mX. 
\end{equation}
The compact, positive definite and self-adjoint operator $L_K$ and the Mercer kernel $K$ can be uniquely determined by each other. The eigenvalues $\{\mu_i\}_{i=1}^\infty$ and eigenfunctions $\{\phi_i\}_{i=1}^\infty$ of $K$ in~\eqref{eq:eigendecomposition} provide an eigendecomposition of $L_K$ as
\begin{equation}
L_K\phi_i=\mu_i\phi_i,\quad i\in \mathbb{N}.
\end{equation}

We introduce the sample analog $L_{K, X}: \Ltwo \rightarrow \bbH$ by 
\begin{equation}\label{eq:sample.analog}
L_{K, X}(f) = \frac{1}{n} \sum_{i = 1}^n f(X_i) K_{X_i}, 
\end{equation}
where $K_{X_i}(\cdot) := K(X_i, \cdot)$. Note that the operator $L_{K, X}$ and the matrix $K(X,X)/n$ share the same eigenvalues. We remark that $L_K$ and $L_{K, X}$ can also be defined on subspaces of $\Ltwo$ such as $C(\mX)$, $C^k[0,1]$, and $H^\alpha[0,1]$.

Let $I$ be the identity operator. We approximate $f_0$ by a function in $\bbH$ defined as follows:
\begin{equation}\label{eq:flambda}
f_{\lambda} = (L_K + \lambda I)^{-1}L_K f_{0},
\end{equation}
which minimizes $ \|f - f_{0}\|_{2}^2 + \lambda \|f\|^2_{\bbH}$ subject to $f \in \bbH.$

The equivalent kernel $\tilde{K}$ defined in \eqref{eq:equivalent.kernel} provides an alternative interpretation of the proximate function $f_\lambda$. Let the integral operator $L_{\tilde{K}}$ be the counterpart of $L_K$ induced by $\tilde{K}$, then we have $f_\lambda = L_{\tilde{K}}f_0$. Similarly, we define the sample analog of $L_{\tilde{K}}$ by
\begin{equation}
L_{\tilde{K}, X}(f) = \frac{1}{n} \sum_{i = 1}^n f(X_i) \tilde{K}_{X_i}, 
\end{equation}
which is also a compact, positive definite, and self-adjoint operator.

A RKHS bound for $L_{K,X}(f)-L_K(f)$ has been available in \cite{smale2005shannon, smale2007learning} based on an integral operator approach. However, such an analysis usually requires the response $Y$ to be uniformly bounded, which excludes nonparametric regression that assumes Gaussian error. \cite{liu2020non} adopted the operator-theoretic approach and extended it to unbounded sampling process and general Mercer kernels, which provided the following non-asymptotic RKHS bound for $\hat{f}_n - f_{\lambda}$ using Hanson-Wright inequality \citep{rudelson2013hanson} and equivalent kernels. For ease of reference we provide it below; see Theorem 1 in \cite{liu2020non} for a proof.

\begin{thm}[Theorem 1 in \cite{liu2020non}]\label{thm:equivalent.bound}
Under Condition (A1), it holds with $\PP_0^{(n)}$-probability at least $1 - n^{-10}$  that
\begin{equation}
\|\hat{f}_n-f_\lambda\|_{\tilde{\bbH}}\leq \frac{\tilde{\kappa}^{-1}C(n,\tilde{\kappa})}{1-C(n,\tilde{\kappa})}\|f_\lambda-f_0\|_\infty+\frac{1}{1-C(n,\tilde{\kappa})}\frac{4\tilde{\kappa}\sigma\sqrt{20\log n}}{\sqrt{n}},
\end{equation}
where $C(n,\tilde{\kappa})=\frac{\tilde{\kappa}^2 \sqrt{20\log n}}{\sqrt{n}} \left(4 + \frac{4 \tilde{\kappa}\sqrt{20\log n}}{3 \sqrt{n}} \right).$
\end{thm}

The theorem above is applicable for noiseless observations when $\sigma = 0$ and $Y=f_0(X) := (f_0(X_1), \ldots, f_0(X_n))^T$. This noiseless case, as will be shown later, is closely related to uncertainty quantification of GP priors, which in turn plays a critical role in establishing the equivalence theory. We next present a corollary based on the noise-free version of Theorem~\ref{thm:equivalent.bound} that zeroes out the second term in the error bound. 
\begin{cor}\label{cor:h.tilde.noiseless}
Suppose the observations are noiseless. Under Condition (A1), by choosing $\lambda$ such that $\tilde{\kappa}^2=o(\sqrt{n/\log n})$, it holds with $\PP_0^{(n)}$-probability at least $1 - n^{-10}$ that
\begin{equation}
\|\hat{f}_n-f_0\|_{\tilde{\bbH}}\leq 2\|f_\lambda-f_0\|_{\tilde{\bbH}},
\end{equation}
and
\begin{equation}
\|\hat{f}_n-f_0\|_{\infty}\leq 2\|f_\lambda-f_0\|_{\infty}.
\end{equation}
\end{cor}

We remark that Theorem~\ref{thm:equivalent.bound} and Corollary~\ref{cor:h.tilde.noiseless} do not require the strong condition $f_0\in\bbH$ that is sometimes assumed in the existing work.

We next consider two function spaces $H^\alpha[0,1]$ and $A^\gamma[0,1]$ and derive the error bounds for $\hat{f}_n$, which are nearly minimax optimal.

\begin{lem}[Theorem 9 in \cite{liu2020non}]\label{lem:minimax.matern}
Suppose $f_0\in H^\alpha[0,1]$ for $\alpha>1/2$. If $K_\alpha$ is used in the GP prior, then it holds with $\PP_0^{(n)}$-probability at least $1-n^{-10}$ that
\begin{equation}
\|\hat{f}_n-f_0\|_2 \lesssim \left(\frac{\log n}{n}\right)^{\frac{\alpha}{2\alpha+1}},
\end{equation}
with the corresponding choice of regularization parameter $\lambda\asymp ({\log n}/{n})^{\frac{2\alpha}{2\alpha+1}}$.
\end{lem}

\begin{lem}\label{lem:minimax.exp}
Suppose $f_0\in A^\gamma[0,1]$ for $\gamma>0$. If $K_\gamma$ is used in the GP prior, then it holds with $\PP_0^{(n)}$-probability at least $1-n^{-10}$ that
\begin{equation}
\|\hat{f}_n-f_0\|_2 \lesssim \frac{\log n}{\sqrt{n}},
\end{equation}
with the corresponding choice of regularization parameter $\lambda\asymp 1/n$.
\end{lem}

We next turn to deriving non-asymptotic bounds for derivatives and posterior variances, which are central to the nonparametric plug-in property of GP priors and the equivalence theory.

\subsection{Convergence rates of derivatives of posterior mean}\label{sec:derivatives}
In this section, we derive error bounds for the derivatives of the posterior mean estimator, which serve as a major component in establishing the aforementioned nonparametric plug-in property. Here we consider $f_0\in H^\alpha[0,1]$ and use $K_\alpha$ in the GP prior with a uniform sampling process as in Section~\ref{sec:operator} to ease presentation.

Let the higher-order analog of the effective dimension for $K_\alpha$ be $\tilde{\kappa}_{\alpha,mm}^2$, $m\in\mathbb{N}_0$, where the subscript $\alpha$ emphasizes the use of $K_{\alpha}$ compared to the general definition in \eqref{eq:high.order.kappa}. Note that we allow $m=0$ in $\tilde{\kappa}_{\alpha,mm}$, which corresponds to $\tilde{\kappa}_{\alpha,00}^2=\tilde{\kappa}_{\alpha}^2={\sup_{x\in[0,1]}\tilde{K}_\alpha(x,x)}$. Likewise, all results in this subsection encompass $m = 0$ as a special case; hence, we unify the study of derivatives and the regression function in this section. 

Lemma~\ref{lem:differentiability.matern} provides the differentiability of $K_\alpha$ as well as the exact order of $\tilde{\kappa}_{\alpha,kk}$ with respect to the regularization parameter $\lambda$.
\begin{lem}\label{lem:differentiability.matern}
If $\alpha>m+\frac{1}{2}$ for $m\in\mathbb{N}_0$, then $\tilde{K}_\alpha\in C^{2m}([0,1]\times[0,1])$ and $\tilde{\kappa}_{\alpha,mm}^2\asymp \lambda^{-\frac{2m+1}{2\alpha}}$. Moreover, we have $\hat{\kappa}^2_{\alpha,01}\lesssim \lambda^{-\frac{1}{\alpha}}$ when $\alpha>1$ and $\hat{\kappa}^2_{\alpha,k+1,k+1}\lesssim \lambda^{-\frac{2k+3}{2\alpha}}$ when $\alpha>k+\frac{3}{2}$.
\end{lem}

The next lemma studies the differentiability of functions in the RKHS $\bbH_\alpha$ induced by $K_\alpha$. It turns out that functions in $\bbH_\alpha$ inherit the differentiability of $K_\alpha$, and the equivalent RKHS norm upper bounds the $L_2$ norm of the derivatives. Note that $\bbH_\alpha$ and $\tilde{\bbH}_\alpha$ consists of the same class of functions, thus sharing the same differentiability property. 
\begin{lem}\label{lem:matern.RKHS.derivative}
If $\alpha>m+\frac{1}{2}$ for $m\in\mathbb{N}_0$, then $f\in C^m[0,1]$ for any $f\in\bbH_\alpha$. Moreover, these exists a constant $C>0$ that does not depend on $\lambda$ such that $\|f^{(m)}\|_2\leq C\tilde{\kappa}_\alpha^{-1}\tilde{\kappa}_{\alpha,mm}\|f\|_{\tilde{\bbH}_\alpha}$ \ for any $f\in\bbH_\alpha$.

\end{lem}

We divide the error $\hat{f}_n^{(k)}-f_0^{(k)}$ into two parts: $\hat{f}_n^{(k)}-f^{(k)}_\lambda$ and $f_\lambda^{(k)}-f_0^{(k)}$. The first part can be tackled by applying Lemma~\ref{lem:matern.RKHS.derivative} to Theorem~\ref{thm:equivalent.bound}. For the second part $f_\lambda^{(k)}-f_0^{(k)}$, we provide the error bound under the $L_\infty$ norm in the following Lemma~\ref{lem:matern.deriv.deterministic}.
\begin{lem}\label{lem:matern.deriv.deterministic}
Suppose $f_0\in H^{\alpha}[0,1]$ for $\alpha>k+1/2$ and $k\in\mathbb{N}_0$. If $K_\alpha$ is used in the GP prior, then it holds
\begin{equation}
\|f_\lambda^{(k)}-f_0^{(k)}\|_\infty \lesssim\lambda^{\frac{1}{2}-\frac{k}{2\alpha}}.
\end{equation}
\end{lem}

When $m=0$, the above three lemmas above provide error bounds for estimating the regression function. In particular, Lemma~\ref{lem:differentiability.matern} implies $\tilde{\kappa}_{\alpha}^2\asymp \lambda^{-\frac{1}{2\alpha}}$, Lemma~\ref{lem:matern.RKHS.derivative} gives $\|f\|_2\leq \|f\|_{\tilde{\bbH}_\alpha}$, and Lemma~\ref{lem:matern.deriv.deterministic} leads to $\|f_\lambda-f_0\|_\infty \lesssim\lambda^{\frac{1}{2}}$.

Finally, we present a non-asymptotic convergence rate of $\hat{f}_n^{(k)}$ under the $L_2$ norm, which is nearly minimax optimal~\citep{stone1982optimal}.
\begin{thm}\label{thm:minimax.diff.matern}
Suppose $f_0\in H^{\alpha}[0,1]$ for $\alpha>k+1/2$ and $k\in\mathbb{N}$. If $K_\alpha$ is used in the GP prior, then it holds with $\PP_0^{(n)}$-probability at least $1-n^{-10}$ that
\begin{equation}
\|\hat{f}_n^{(k)}-f_0^{(k)}\|_2\lesssim \left(\frac{\log n}{n}\right)^{\frac{\alpha-k}{2\alpha+1}},
\end{equation}
with the corresponding choice of regularization parameter $\lambda\asymp ({\log n}/{n})^{\frac{2\alpha}{2\alpha+1}}$.
\end{thm}

\begin{rem}
Compared with Lemma~\ref{lem:minimax.matern}, we can see that given the smoothness level of the regression function, the rate-optimal estimation of $f$ and its derivatives share the same choice of $\lambda$. Thus, KRR adapts to the order of the derivative to be estimated, which carries over to the counterpart nonparametric Bayesian procedure as a result of the equivalence theory. 
\end{rem}

\subsection{Convergence rates of posterior variance}\label{sec:posterior.variance}
In this section, we study the convergence rate of posterior variance $\hat{V}_n(\bx)$ of $\Pi_n(\cdot\mid\data)$, which plays an important role not only in deriving the equivalence theory but also in uncertainty quantification of nonparametric Bayes.

The posterior covariance $\hat{V}_n(\bx,\bx')$ can be viewed as the bias of a noise-free KRR estimator~\citep{yang2017frequentist}. To see this, let $\hat{K}_{\bx'}$ be the KRR estimator with noiseless observations of $K_{\bx'}$ at $\{X_i\}_{i=1}^n$, i.e.,
\begin{equation}\label{eq:noiseless.krr}
\hat{K}_{\bx'} = \underset{f \in \bbH}{\arg\min} \left\{\frac{1}{n}\sum_{i = 1}^n (K(X_i,\bx') - f(X_i))^2 + \lambda \|f\|^2_{\bbH}\right\}.
\end{equation} 
Then we have $\hat{K}_{\bx'}(\cdot):=K(\cdot, X)[K(X, X) + n \lambda \bI_n]^{-1} K(X, \bx')$. Substituting $\hat{K}_{\bx'}$ into \eqref{eq:variance} yields
\begin{equation}\label{eq:posterior.variance}
\begin{split}
\sigma^{-2}n\lambda \hat{V}_n(\bx,\bx')&=K(\bx, \bx') - K(\bx, X)[K(X, X) + n \lambda \bI_n]^{-1} K(X, \bx')\\
&=K_{\bx'}(\bx)-\hat{K}_{\bx'}(\bx),
\end{split}
\end{equation}
where $K_{\bx'}(\cdot):=K(\cdot,\bx')$. We next provide a non-asymptotic bound for the posterior variance $\hat{V}_n(\bx)$. 
\begin{thm}\label{thm:equivalent.sigma}
Under Condition (A1), by choosing $\lambda$ such that $\tilde{\kappa}^2=o(\sqrt{n/\log n})$, it holds with $\PP_0^{(n)}$-probability at least $1 - n^{-10}$  that
\begin{equation}
\|\hat{V}_n\|_\infty\leq \frac{2\sigma^2\tilde{\kappa}^2}{n}.
\end{equation}
\end{thm}

\subsection{Convergence rates of posterior variances of derivatives of Gaussian processes}
In this section, we present a non-asymptotic bound for the posterior variance of derivatives of Gaussian process, i.e., the variance of $\Pi_{n,k}(\cdot\mid\data)$, providing tools for establishing the nonparametric plug-in property in Theorem~\ref{thm:deriv.contraction.l2}. For simplicity we focus on the one-dimensional case where $\mX=[0,1]$, but point out that our results can be extended to multivariate cases straightforwardly.

If $K\in C^{2k}([0,1],[0,1])$, the posterior covariance of the $k$-th derivative of $f$, denoted by $\tilde{V}^k_n(x, x')$, is given in \eqref{eq:deriv.variance}. By comparing \eqref{eq:variance} and \eqref{eq:deriv.variance}, we can equivalently rewrite $\tilde{V}^k_n(x, x')$ as 
\begin{equation}
\tilde{V}^k_n(x, x') = \partial^k_{x}\partial^k_{x'}\hat{V}_n(x,x').
\end{equation}
Therefore, the posterior covariance of derivative of GP is exactly the mixed derivative of the posterior covariance of the original GP. This is not surprising as the differential operator is linear. We write $\tilde{V}^{k}_n(x):=\tilde{V}^k_n(x,x)$ for the posterior variance. We caution that, however, $\tilde{V}^{k}_n(x)$ may not be obtained by taking the derivatives of $\hat{V}_n(x)$ since differentiation and evaluation at $x=x'$ may not be exchangeable.

Finally, we provide a non-asymptotic error bound for $\tilde{V}^{k}_n(x)$.

\begin{thm}\label{thm:equivalent.sigma.deriv}
Suppose $K\in C^{2k}([0,1],[0,1])$. Under Condition (A1), by choosing $\lambda$ such that $\tilde{\kappa}^2=o(\sqrt{n/\log n})$, it holds with $\PP_0^{(n)}$-probability at least $1 - n^{-10}$ that
\begin{equation}
\|\tilde{V}^{k}_{n}\|_\infty \leq \frac{2\sigma^2\tilde{\kappa}_{kk}^2}{n}.
\end{equation}
\end{thm}

\section{Simulation}\label{sec:simulation}

We carry out simulations to assess the finite sample performance of the proposed nonparametric plug-in procedure for function derivatives; we also evaluate Gaussian process regression for estimating the regression function as a special example when the derivative order is zero. 

We consider the true function $f_0(x)=\sqrt{2}\s i^{-4}\sin i\cos[(i-1/2)\pi x]$, $x\in[0,1]$, which has H\"older smoothness level $\alpha=3$. We simulate $n$ observations from the regression model $Y_i=f_0(X_i)+\varepsilon_i$ with $\varepsilon_i\sim N(0, 0.1)$ and $X_i\sim \text{Unif}[0,1]$. We consider three sample sizes 100, 500, and 1000, and replicate the simulation 1000 times. 

For Gaussian process priors, we vary the choice of covariance kernels; in particular, we use the Mat\'ern kernel, squared exponential (SE) kernel and second-order Sobolev kernel, which are given by
\begin{align}
K_{\text{Mat},\nu}(x,x')&=\frac{2^{1-\nu}}{\Gamma(\nu)}\left(\sqrt{2\nu}|x-x'|\right)^\nu B_\nu\left(\sqrt{2\nu}|x-x'|\right),\\
K_{\text{SE}}(x,x')&=\exp(-(x-x')^2),\\
K_{\text{Sob}}(x,x')&=1+xx'+\min\{x,x'\}^2(3\max\{x, x'\}-\min\{x, x'\})/6.
\end{align}
Here $B_\nu(\cdot)$ is the modified Bessel function of the second kind with $\nu$ being the smoothness parameter to be determined. For the Mat\'ern kernel, it is well known that the eigenvalues of $K_{\text{Mat},\nu}$ decay at a polynomial rate, that is,  $\mu_i\asymp i^{-2(\nu+1/2)}$ for $i\in\mathbb{N}$. 

We compare various Gaussian process priors with a random series prior using B-splines. B-splines are widely used in nonparametric regression \citep{james2009functional,wang2020functional}, and theoretical properties of using the B-spline prior with normal basis coefficients to estimate function derivatives have been recently studied in \cite{yoo2016supremum}. The implementation of this B-spline prior follows \cite{yoo2016supremum}. In particular, for any $x\in[0,1]$, let $ b_{J,4}(x)=(B_{j,4}(x))_{j=1}^J$ be a B-spline of order $4$ and degrees of freedom $J$ with uniform knots. The prior on $f$ is given as $f(x)=b_{J,4}(x)^T \beta$ with each entry of $\beta$ following $N(0, \sigma^2)$ independently. The unknown variance $\sigma^2$ is estimated by its MMLE $\hat{\sigma}_n^2=n^{-1}Y^T(B B^T+\bI_n)^{-1}Y$, where $B=(b_{J,4}(X_1),\ldots,b_{J,4}(X_n))^T$. The number of interior knots $N=J-4$ is determined using leave-one-out cross validation. We adopt the same strategy of leave-one-out cross validation to select the degree of freedom $\nu$ in $K_{\text{Mat},\nu}$ for a fair comparison. In particular, we consider $\nu$ from the set $\{2, 2.5, 3, 3.5, 4\}$ and vary $N$ from 1 to 10. We also include a Mat\'ern kernel with the oracle value $\nu=2.5$ that matches the smoothness level of $f_0$. For Gaussian process priors, the regularization parameter $\lambda$ and unknown $\sigma^2$ are determined by empirical Bayes through maximizing the marginal likelihood.

For each method, we evaluate the posterior mean $\hat{f}_n$ and $\hat{f}_n'$ at 100 equally spaced points in $[0,1]$, and calculate the root mean square error (RMSE) between the estimates and the true functions:
\begin{equation}
\mathrm{RMSE} = \sqrt{\frac{1}{100} \sum_{t = 0}^{99} \{\hat{s}(t/99) - s(t /99)\}^2}, 
\end{equation}
where $\hat{s}$ is the estimated function ($\hat{f}_n$ or $\hat{f}_n'$) and $s$ is the true function ($f_0$ or $f_0'$). 

Table~\ref{table:RMSE} reports the average RMSE of all methods over 1000 simulations for $f_0$ and $f_0'$. Clearly the RMSE of all methods steadily decreases as the sample size $n$ increases. The squared exponential kernel and Mat\'ern kernel are the two leading approaches for all sample sizes and for both $f_0$ and $f_0'$. While the difference between various methods for $f_0$ tends to vanish when the sample size increases to $n = 1000$, the performance gap in estimating $f_0'$ is more profound. In particular, compared to the squared exponential kernel, the Sobolev kernel increases the average RMSE by nearly 60\% from 1.53 to 2.41, while the increase becomes more than twofold for the B-spline method (from 1.53 to 3.47). We report the median RMSEs for B-splines in the last row of Table~\ref{table:RMSE} as we notice considerably large RMSEs in a proportion of simulations, which improves the summarized RMSEs to close to the Sobolev kernel for both $f_0$ and $f_0'$ when $n = 1000$. The leave-one-out cross validation method appears to work well for the Mat\'ern kernel as it gives almost identical RMSEs to the Mat\'ern kernel with oracle $\nu$, for both $f_0$ and $f_0'$ at $n = 100, 500$, although there are minor differences at $n = 1000$. 

\setlength{\tabcolsep}{4pt}
\begin{table}
\centering
\renewcommand\arraystretch{1.5}
\caption{RMSE of estimating $f_0$ and $f_0'$, averaged over 1000 simulations.  Mat\'ern, SE, Sobolev, and B-splines: four methods in comparison, which are the Mat\'ern kernel, squared exponential kernel, second-order Sobolev kernel, and random series prior using B-splines. Mat\'ern*: Mat\'ern kernel with oracle $\nu$; B-splines*: median RMSEs of the B-spline method. Standard errors are provided in parentheses for all methods except B-spline*.  \label{table:RMSE}} 
\begin{tabular}{ccccccc}
	\toprule
	\multicolumn{1}{l}{}            & \multicolumn{3}{c}{$f_0$}                    & \multicolumn{3}{c}{$f_0'$}              \\ \cmidrule(lr){2-4} \cmidrule(lr){5-7}
	& $n = 100$           & 500           & 1000          & $n = 100$        & 500         & 1000        \\ \cmidrule{1-1} \cmidrule(lr){2-4} \cmidrule(lr){5-7}
	Mat\'ern* & 0.515 (0.006) & 0.251 (0.003) & 0.193 (0.002) & 2.74 (0.03) & 1.69 (0.02) & 1.48 (0.02) \\
	Mat\'ern          & 0.512 (0.006) & 0.252 (0.003) & 0.201 (0.002) & 2.74 (0.04) & 1.73 (0.02) & 1.68 (0.03) \\
	SE                               & 0.494 (0.007) & 0.254 (0.003) & 0.194 (0.002) & 2.41 (0.03) & 1.75 (0.02) & 1.53 (0.01) \\
	Sobolev                          & 0.507 (0.006) & 0.289 (0.002) & 0.241 (0.002) & 3.06 (0.03) & 2.53 (0.01) & 2.41 (0.01) \\
	B-splines          & 0.770 (0.009) &   0.343 (0.004)  &   0.250 (0.003)         & 5.92 (0.13) &   4.30 (0.12) &  3.47 (0.10)        \\
	B-splines*                       & 0.751         & 0.329         & 0.242         & 4.99        & 3.18        & 2.43        \\ \bottomrule
\end{tabular}
\end{table}

Figure~\ref{fig:boxplot} presents the boxplot of RMSEs for each method at $n = 1000$, which confirms the numerical summaries in Table~\ref{table:RMSE}. While all methods have unusually large RMSEs in some simulations, the B-spline prior exhibits more variability, particularly in estimating $f_0'$, with median performance similar to the Sobolev kernel. The squared exponential kernel appears to be more stable than the Mat\'ern kernel, while achieving slightly better median RMSEs at this large sample size $n = 1000$. 

\begin{figure}[H]
\centering
\begin{tabular}{cc}
	\includegraphics[width=0.48\linewidth]{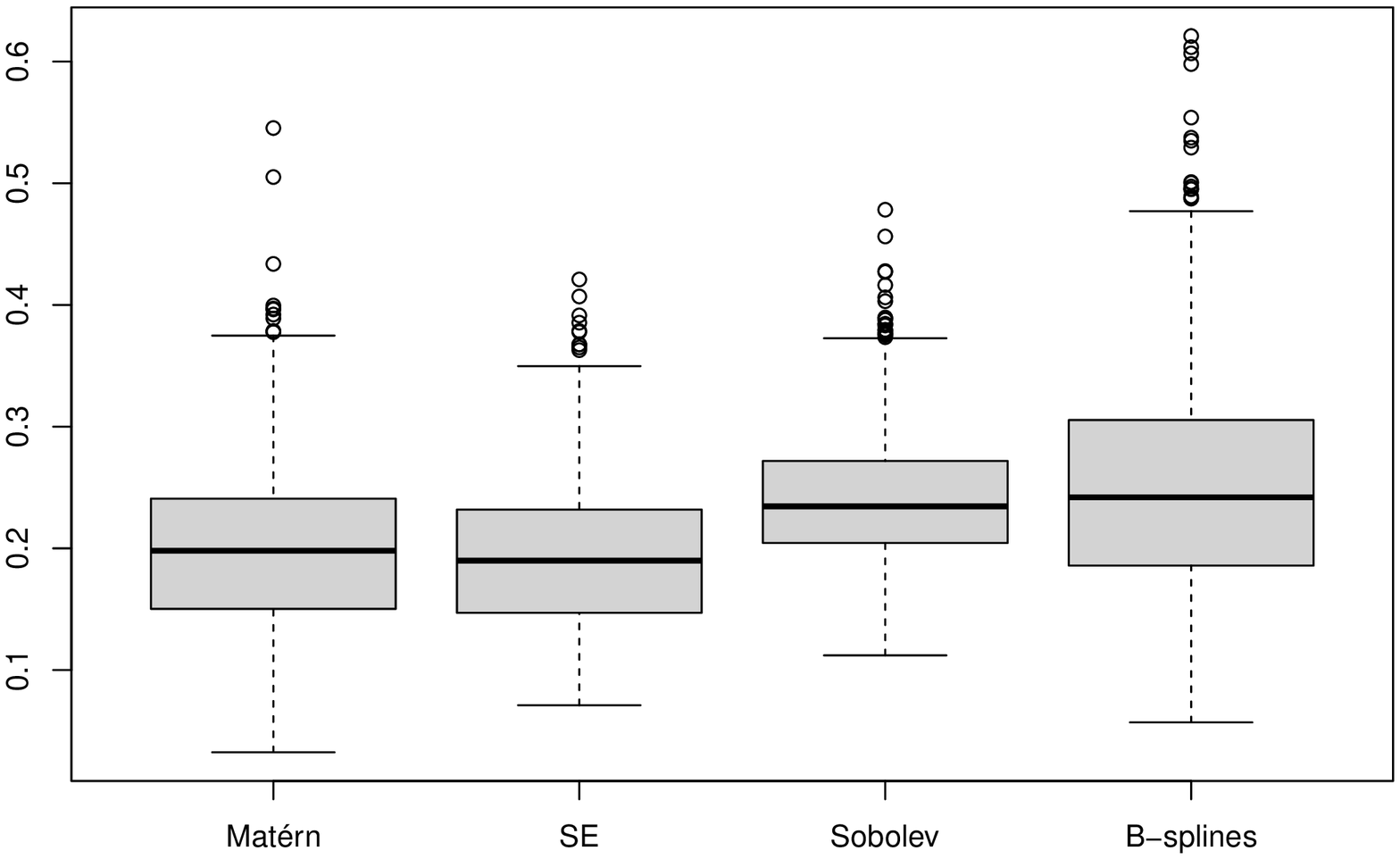}     & \includegraphics[width=0.48\linewidth]{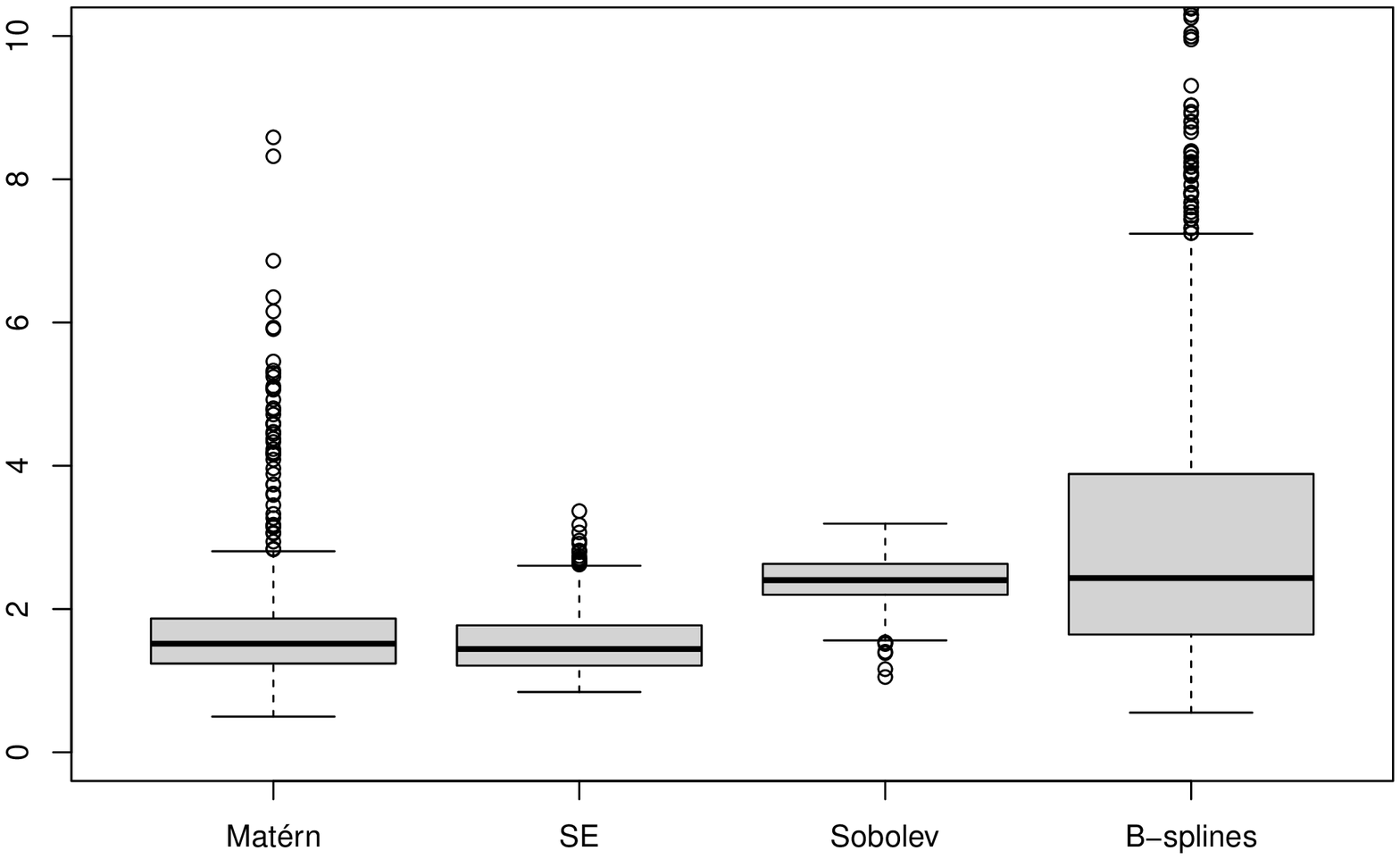} 
\end{tabular}
\caption{Boxplots of RMSEs for estimating $f_0$ (left) and $f_0'$ (right) at $n=1000$. For the plot on the right, the maximum $y$-axis is set to be around 10 for a better visualization; the maximum RMSE for B-splines is 26.46.} \label{fig:boxplot}
\end{figure}

We next choose two representative simulations to visualize the estimates of $f_0$ and $f_0'$ in Figure~\ref{fig}, where the dotted line stands for the posterior mean $\hat{f}_n$ and dashes lines for the 95\% simultaneous $L_\infty$ credible bands. For the B-spline prior, we use the default setting as in \cite{yoo2016supremum} by specifying the inflation factor $\rho=0.5$. Note that these credible bands have fixed width by construction.

\newcommand{\scale}{0.225}
\begin{figure}[h]
\centering 
\begin{tabular}{c c c @{\hskip0.01pt} c @{\hskip0.01pt} c @{\hskip0.01pt} c}
\vspace{0.2in}
& &\ \ Mat\'ern
	& \ SE & \ Sobolev & \ B-splines \\ 
\begin{sideways}
\rule[0pt]{-0.45in}{0pt} Simulation 1
\end{sideways} \quad &

\begin{sideways}
\rule[0pt]{0.45in}{0pt} $f_0$
\end{sideways} &

{\includegraphics[width = \scale\linewidth]{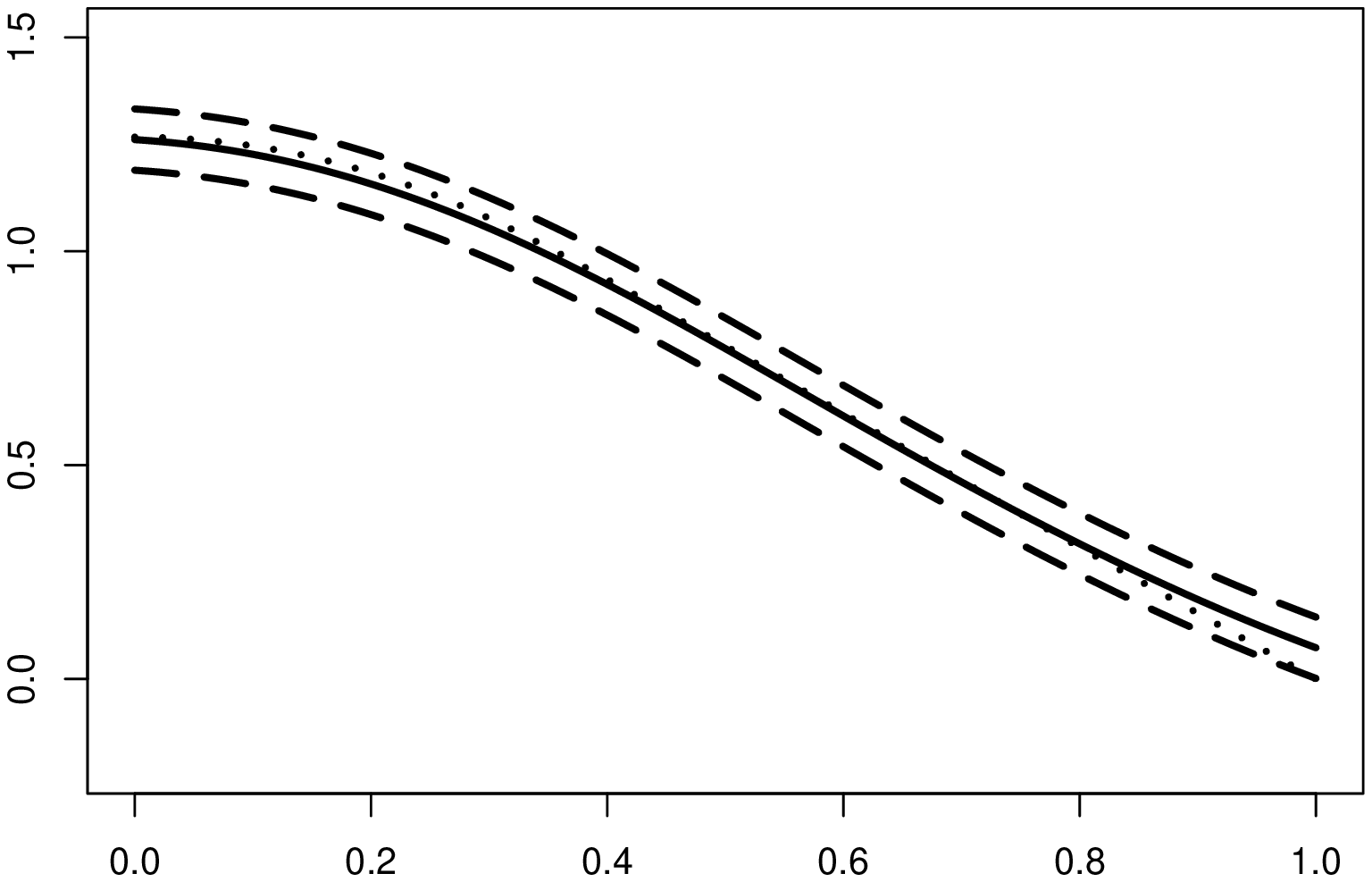}}  & 
{\includegraphics[width = \scale\linewidth]{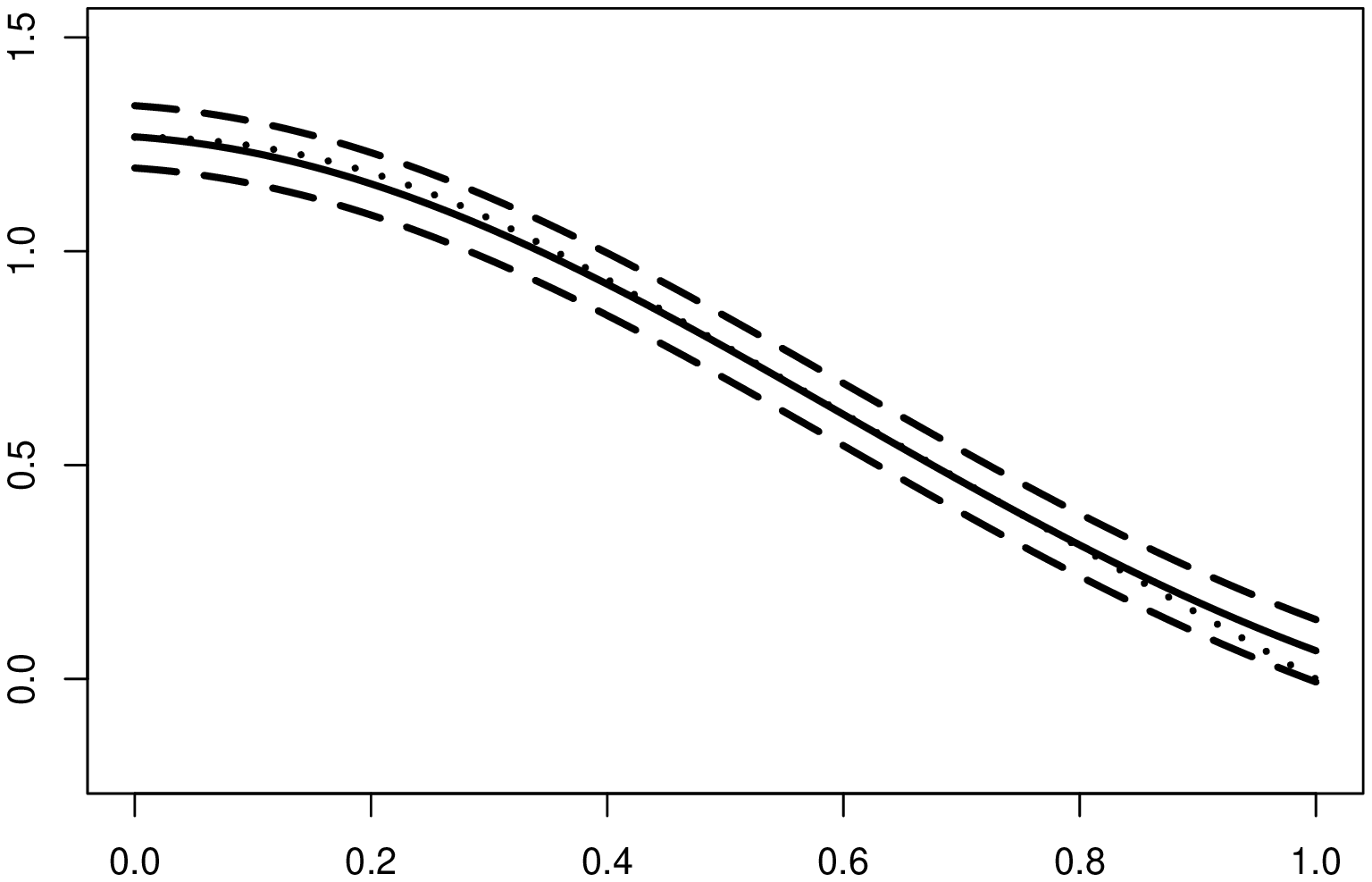}} &
{\includegraphics[width = \scale\linewidth]{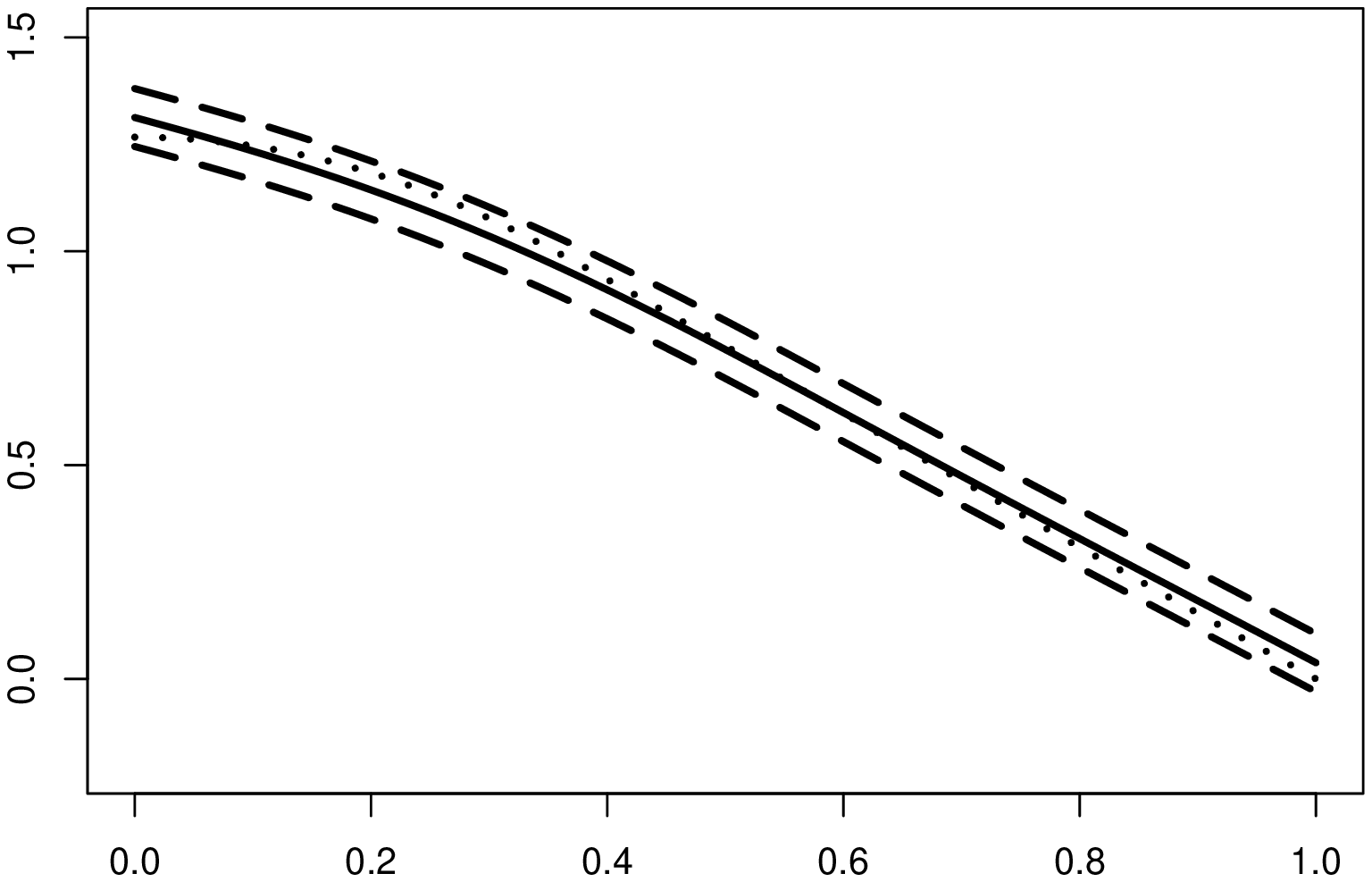}} &
{\includegraphics[width = \scale\linewidth]{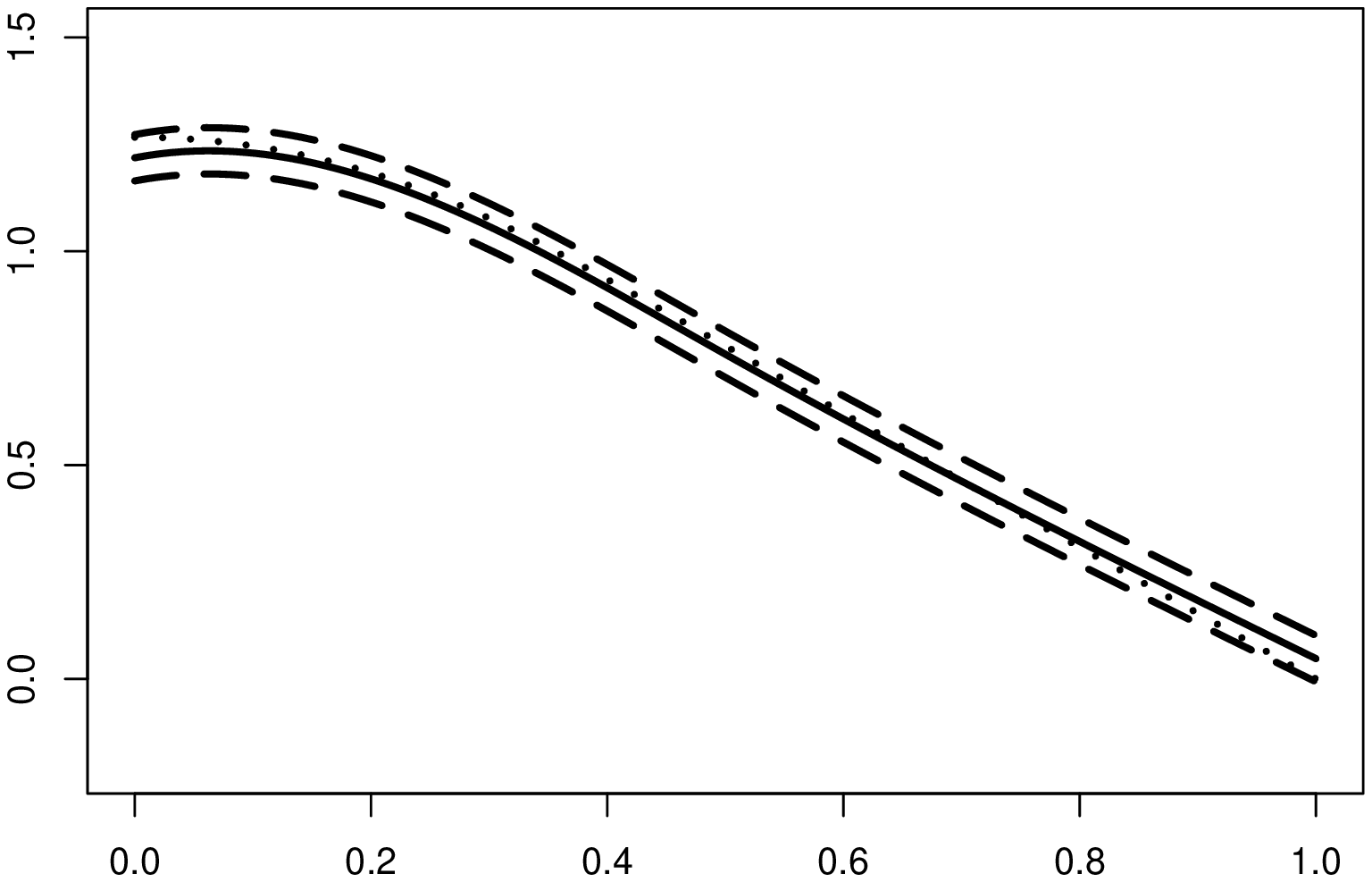}} \\

\vspace{0.2in}

&
\begin{sideways}
\rule[0pt]{0.45in}{0pt} $f_0'$
\end{sideways} &
\includegraphics[width = \scale\linewidth]{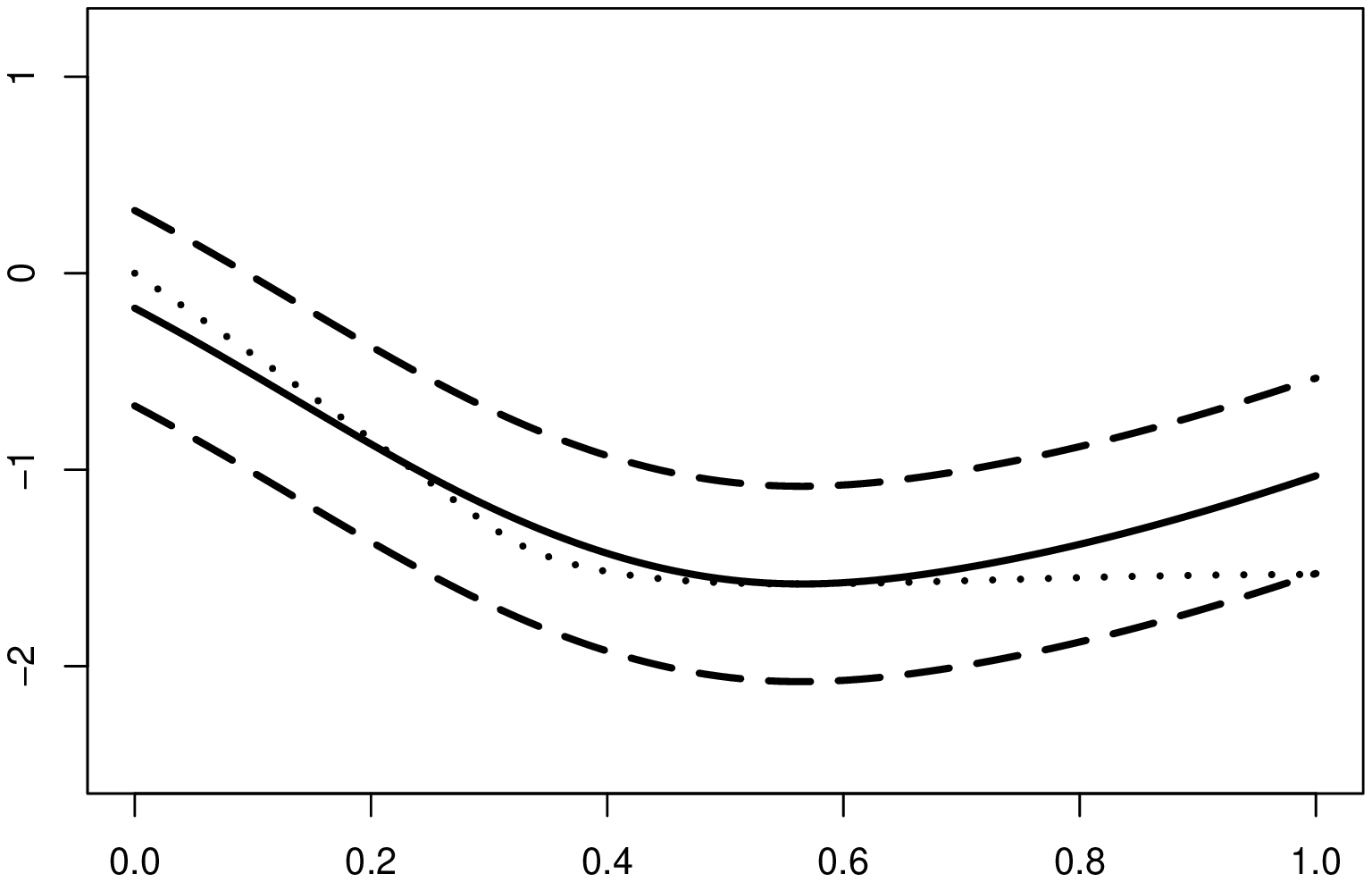} &
\includegraphics[width = \scale\linewidth]{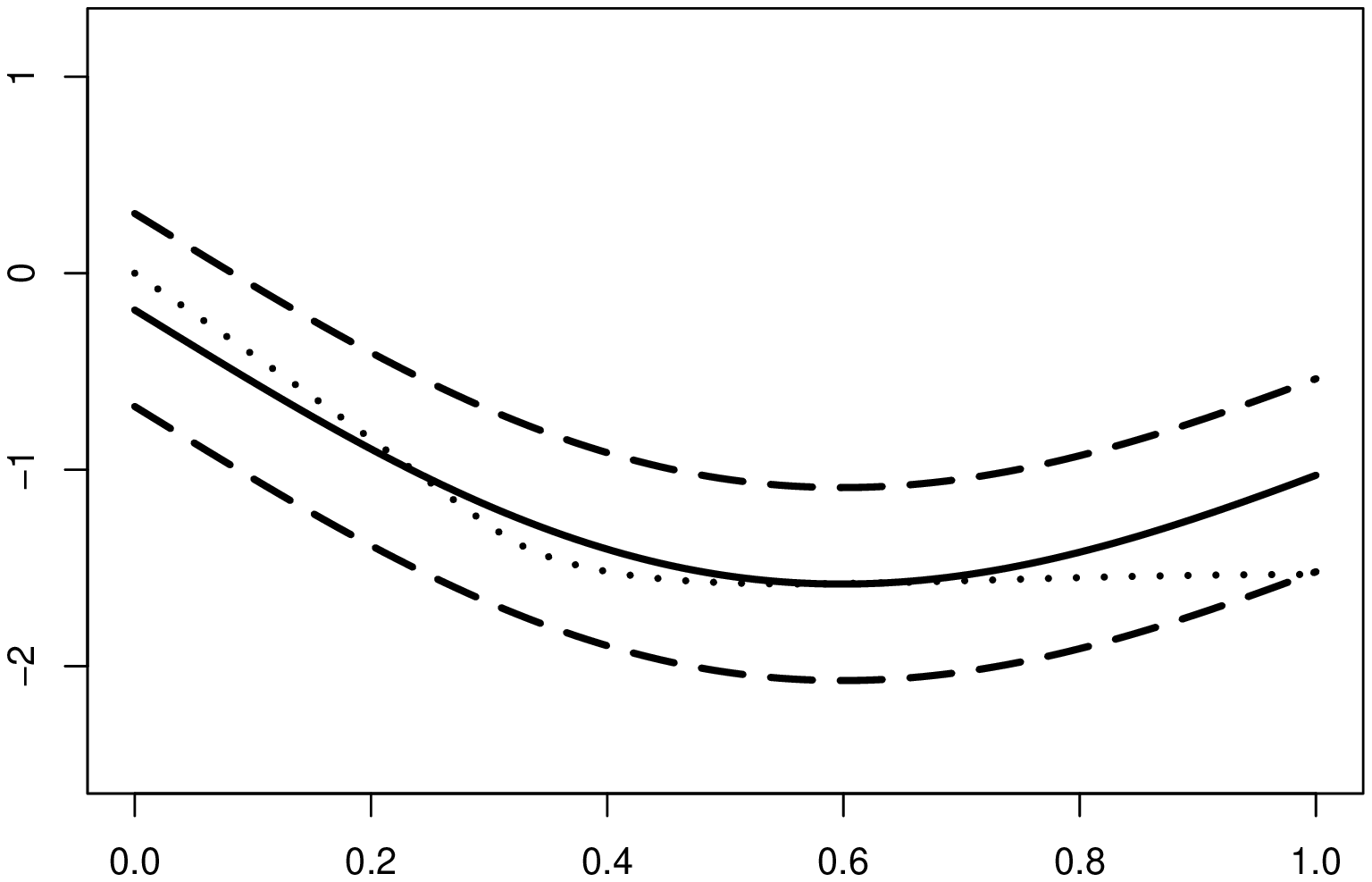} &
\includegraphics[width = \scale\linewidth]{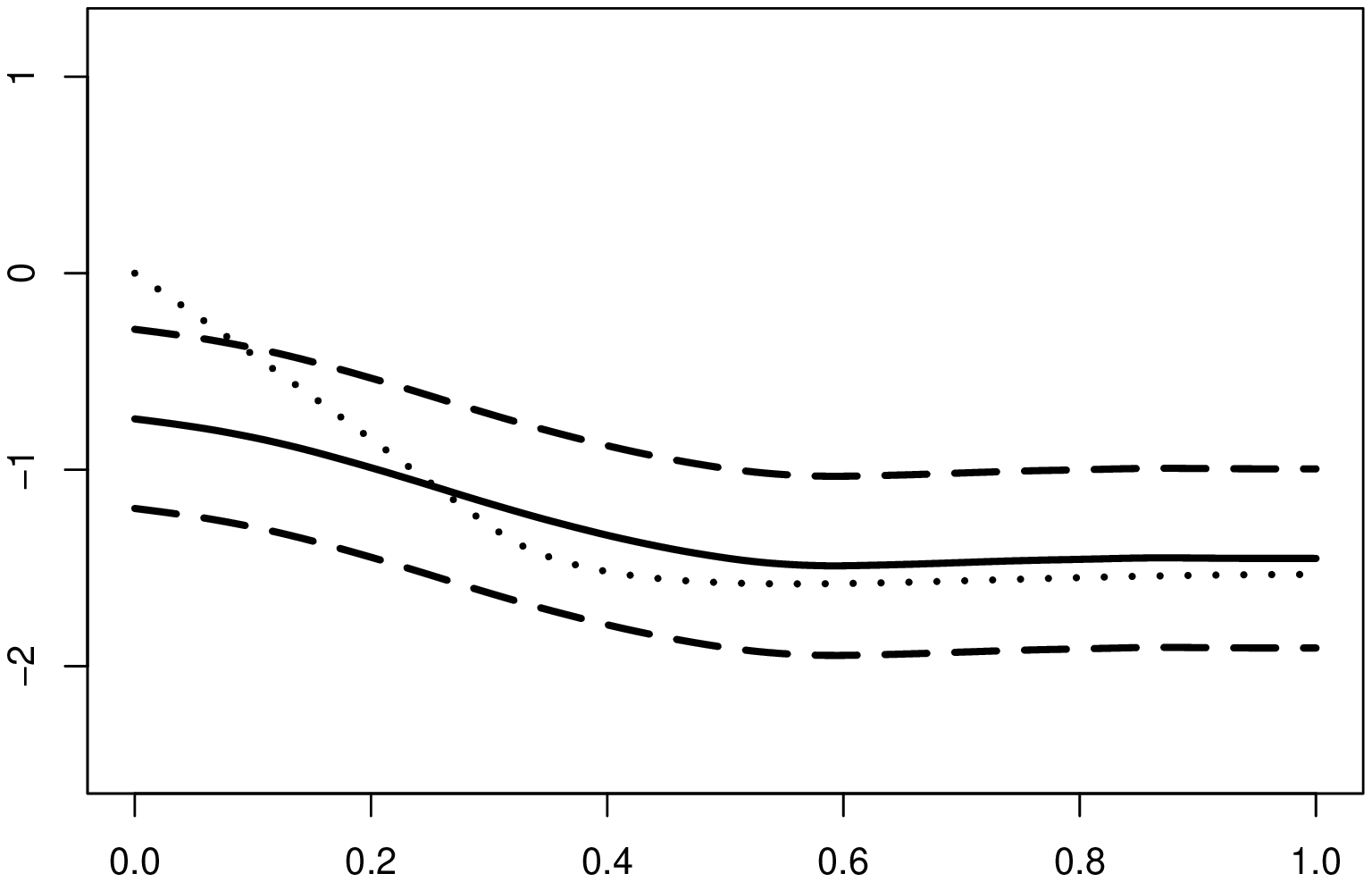} &
\includegraphics[width = \scale\linewidth]{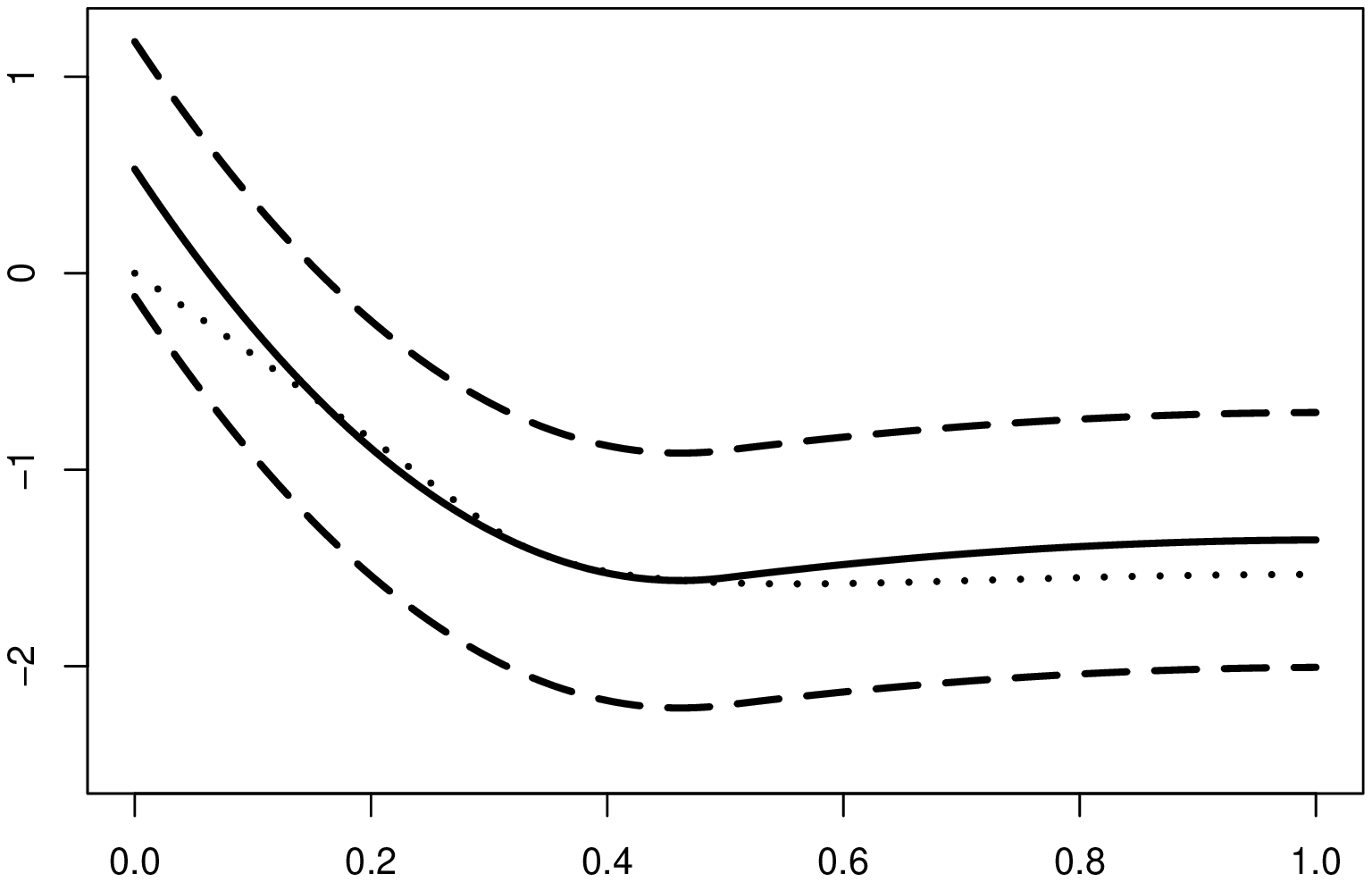} \\

\begin{sideways}
\rule[0pt]{-0.45in}{0pt} Simulation 2
\end{sideways} &
\begin{sideways}
\rule[0pt]{0.45in}{0pt} $f_0$
\end{sideways} &
\includegraphics[width = \scale\linewidth]{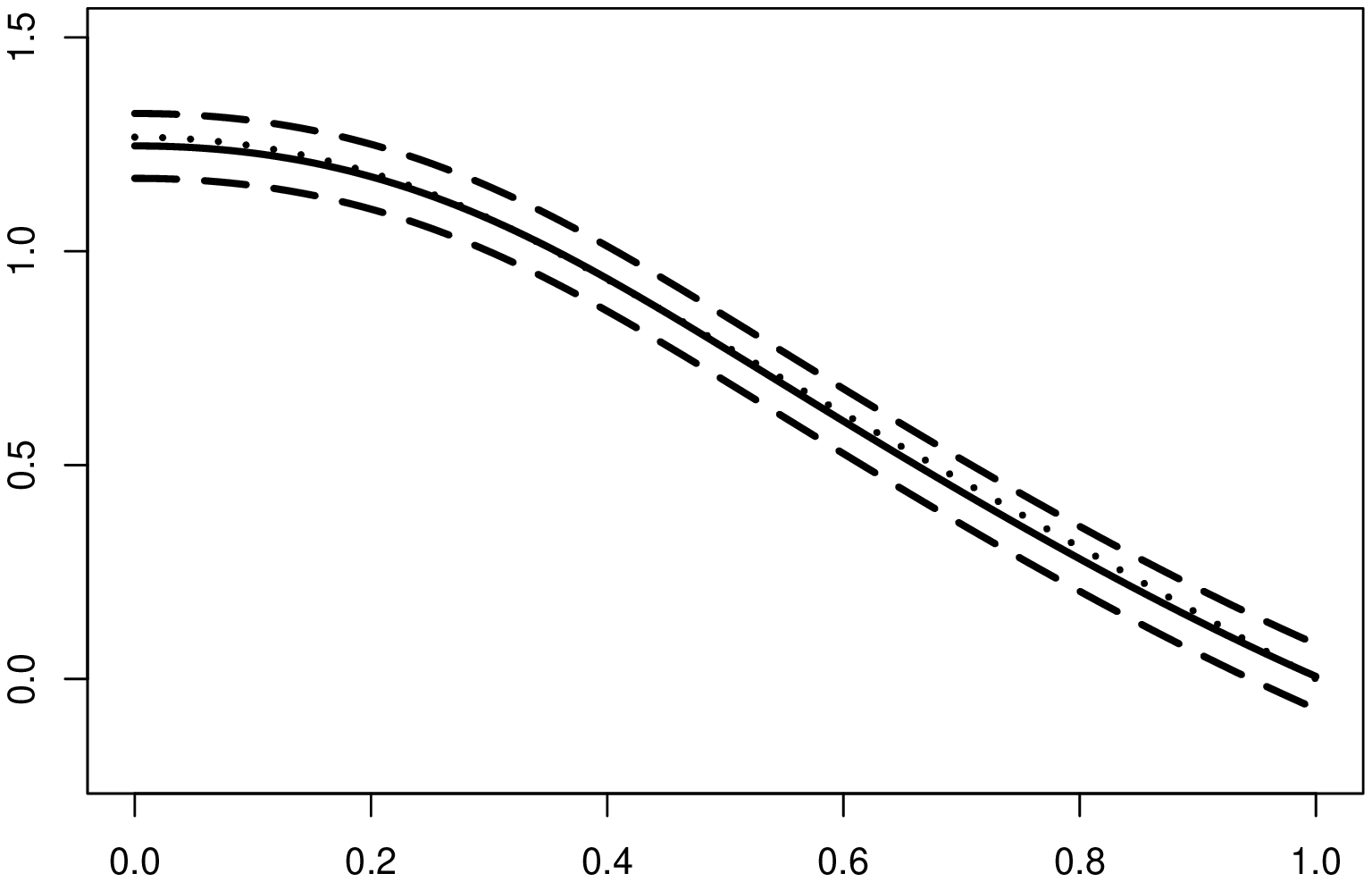} &
\includegraphics[width = \scale\linewidth]{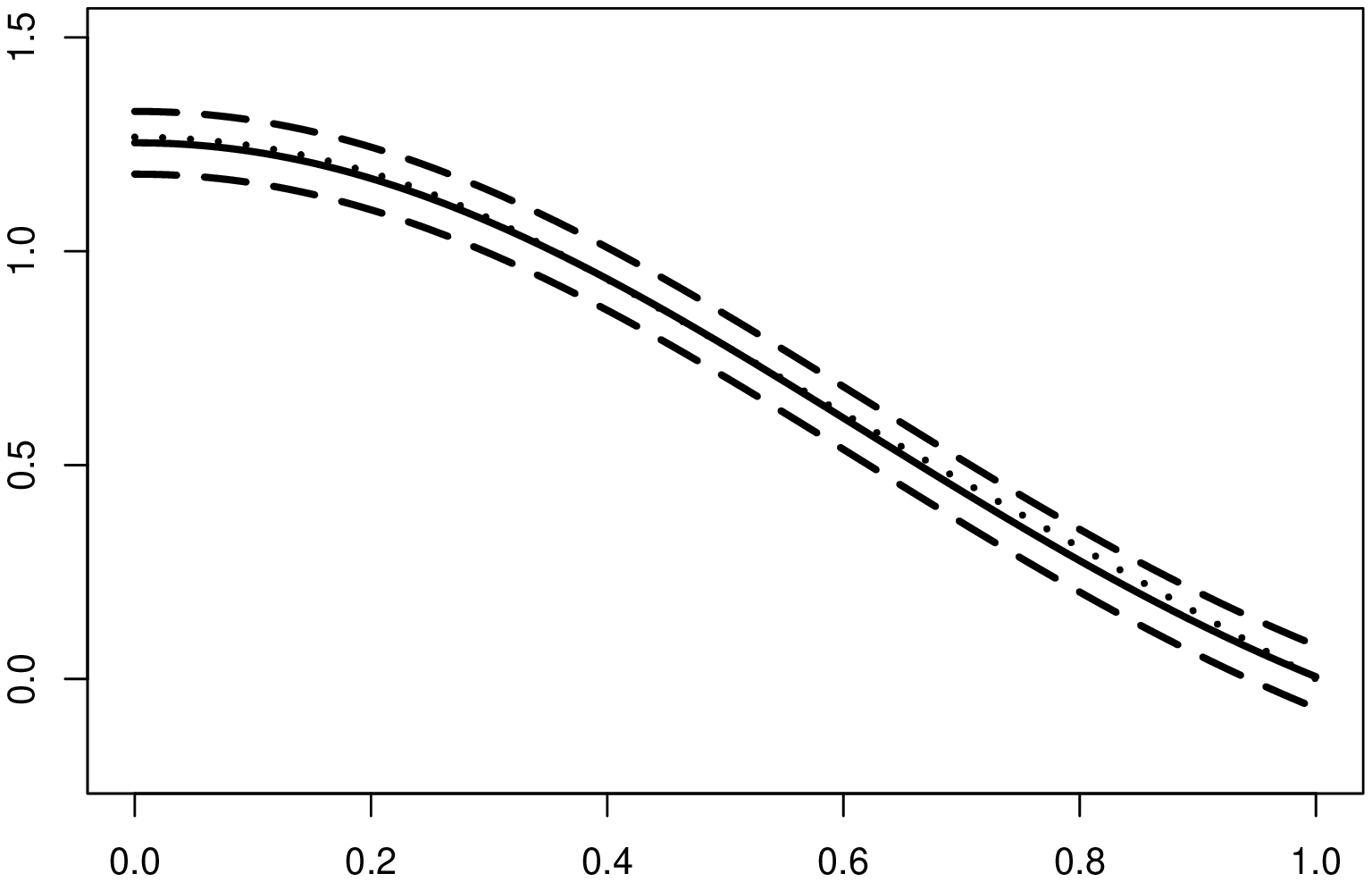} &
\includegraphics[width = \scale\linewidth]{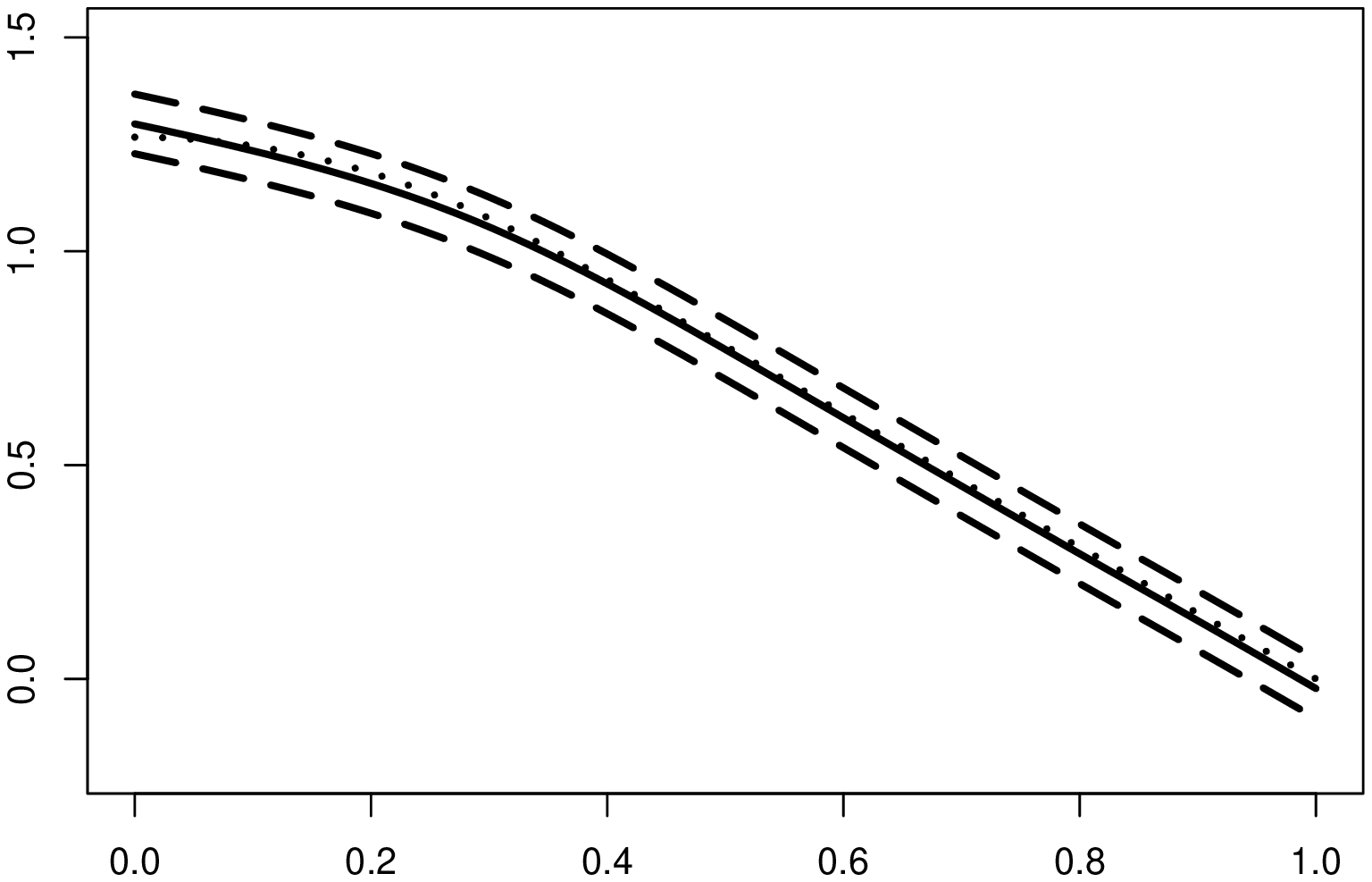} &
\includegraphics[width = \scale\linewidth]{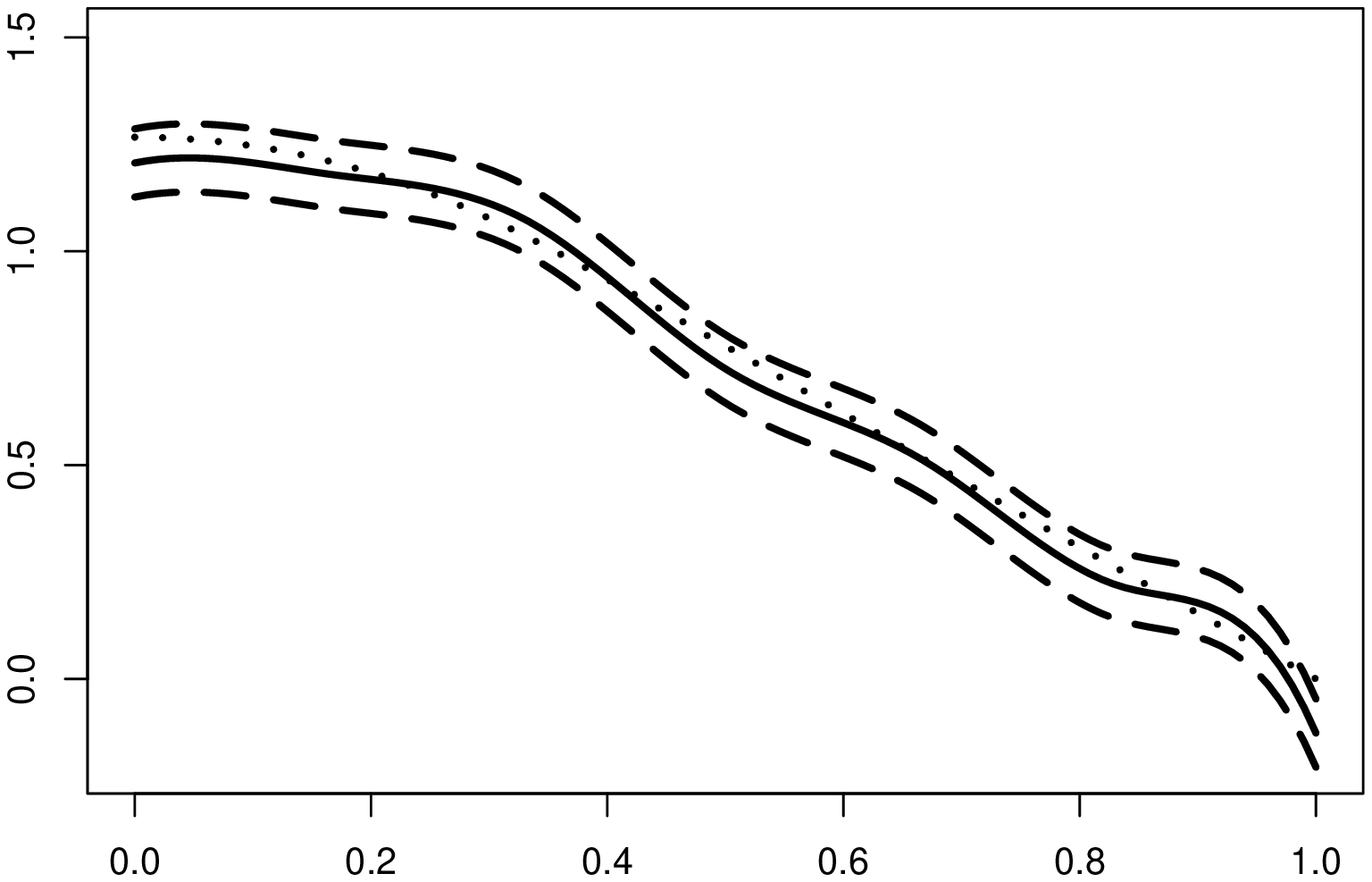} \\ 

&
\begin{sideways}
\rule[0pt]{0.45in}{0pt} $f_0'$
\end{sideways} &
\includegraphics[width = \scale\linewidth]{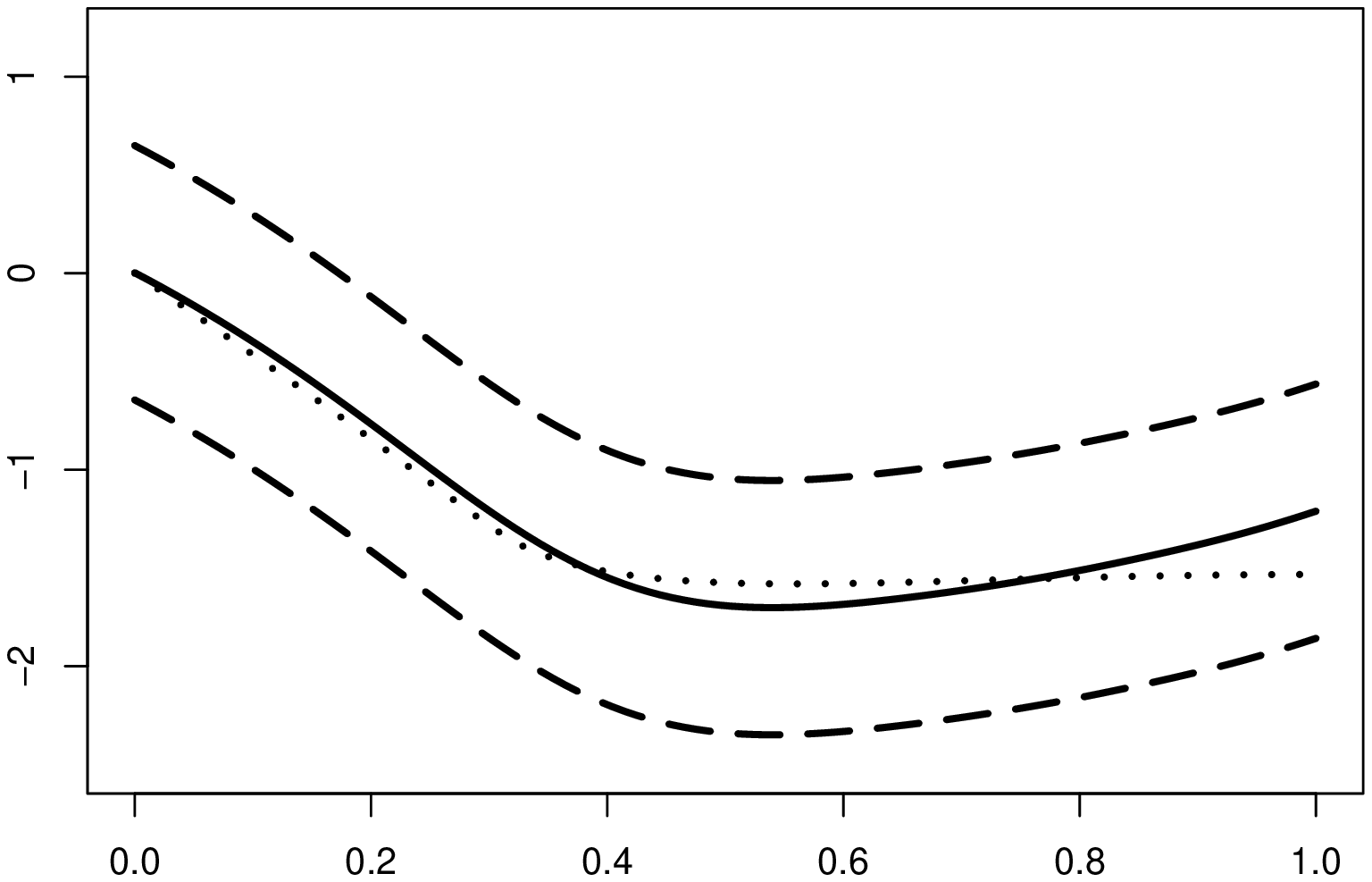} &
\includegraphics[width = \scale\linewidth]{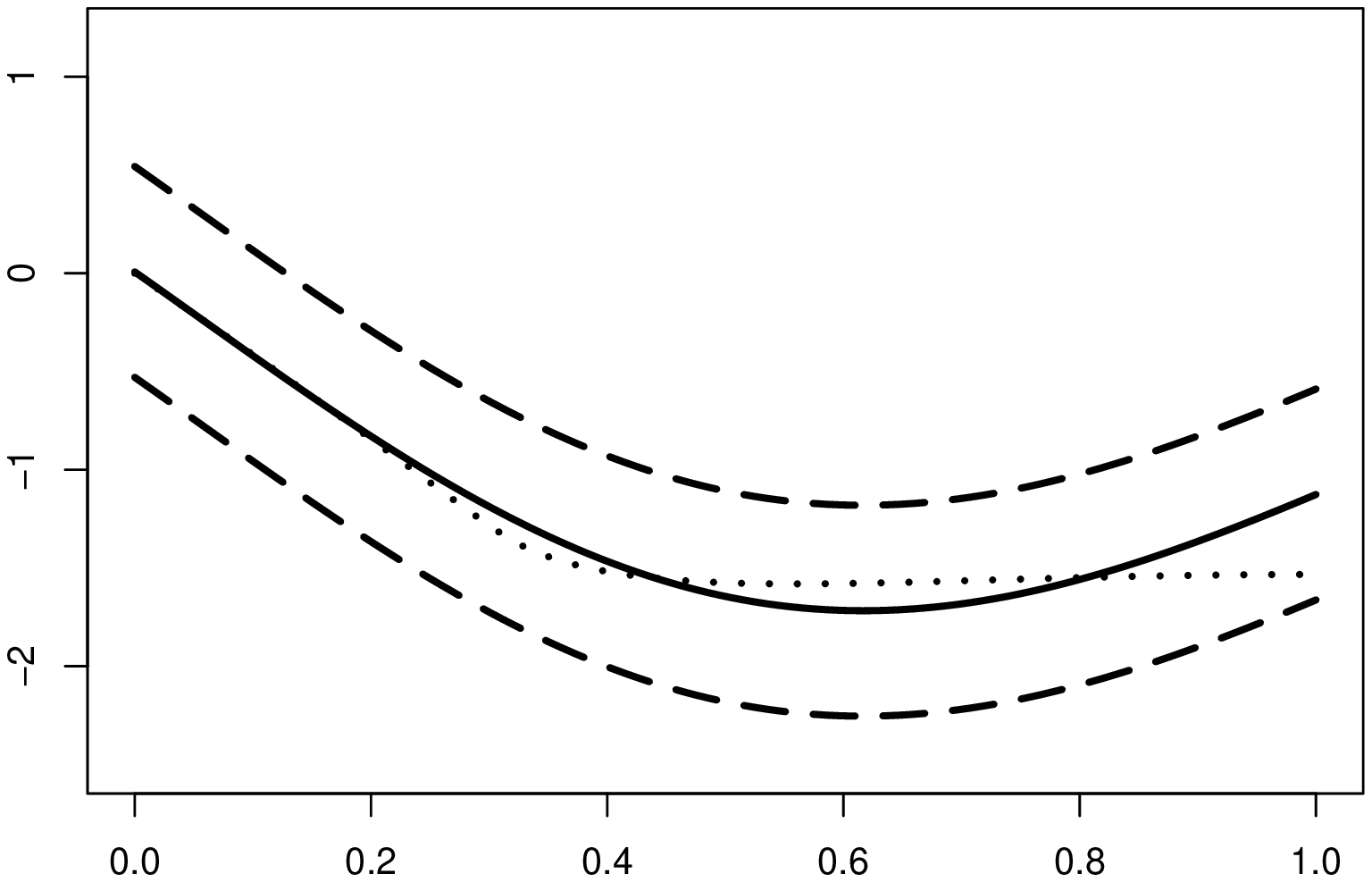} &
\includegraphics[width = \scale\linewidth]{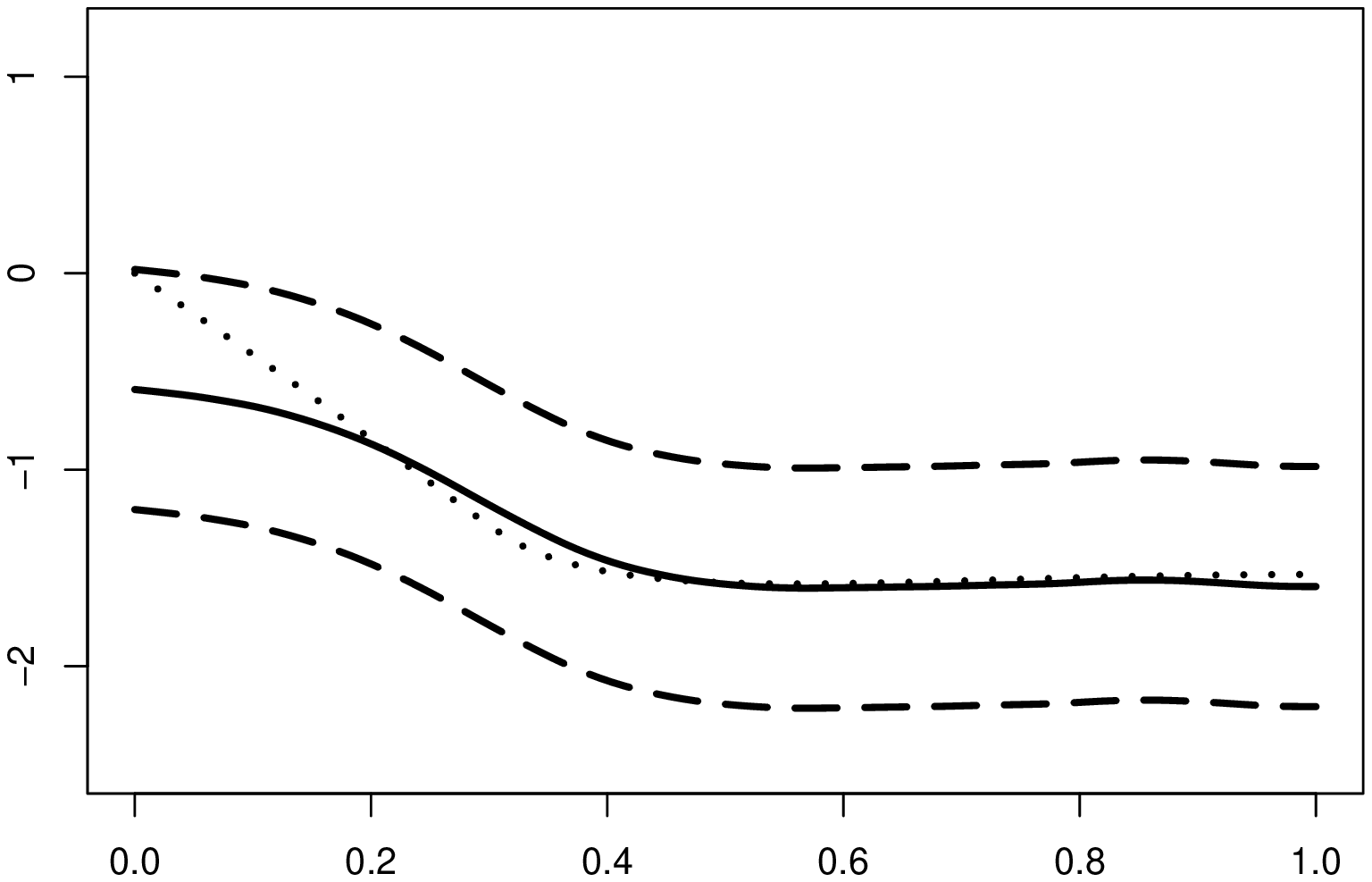} &
\includegraphics[width = \scale\linewidth]{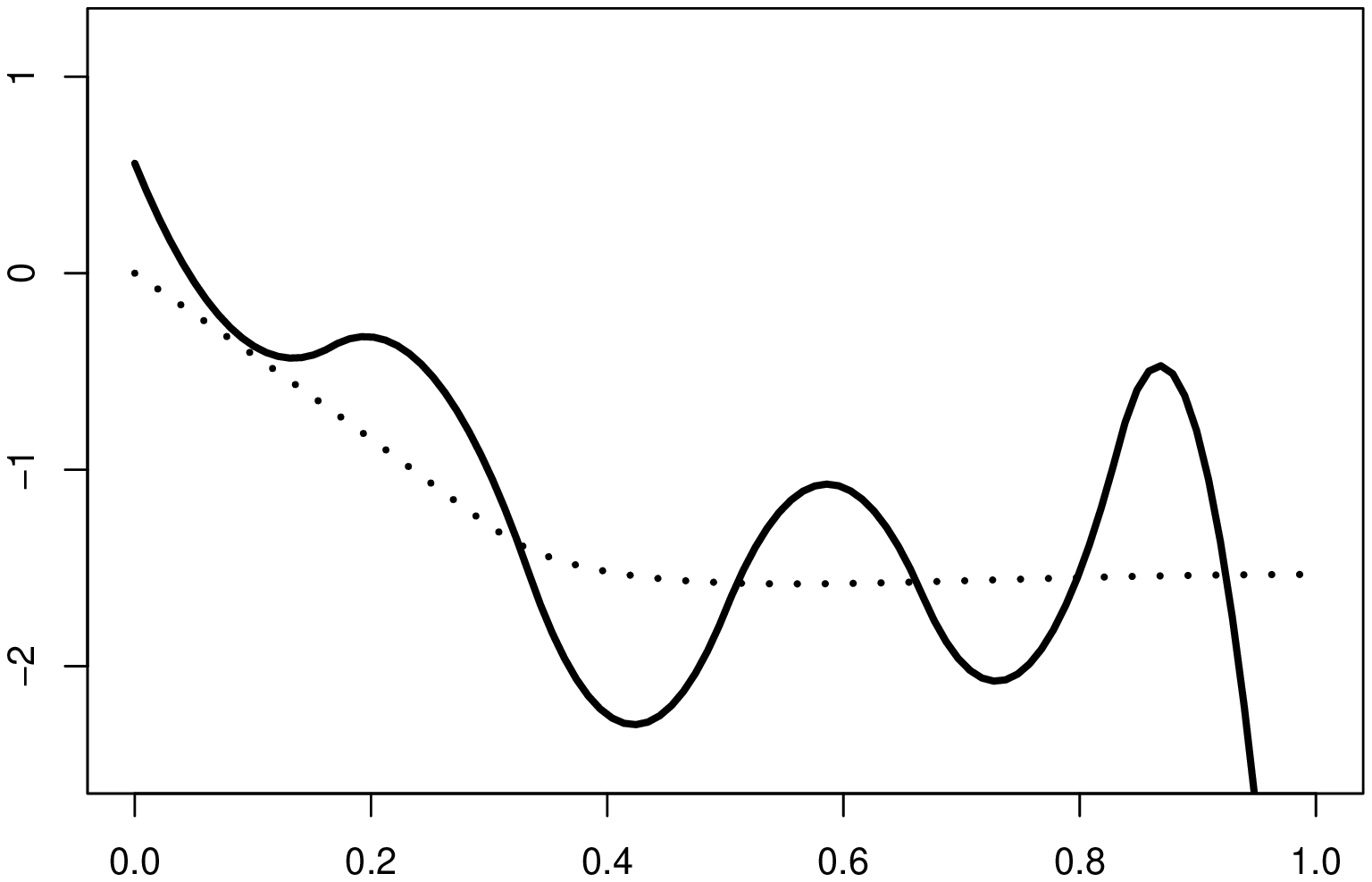} \\
\end{tabular} 
\caption{Results of two simulations with $n=1000$. Row 1 and 2: estimations of $f_0$ and $f_0'$ in Simulation 1. Row 3 and 4: estimations of $f_0$ and $f_0'$ in Simulation 2. Dots: posterior mean $\hat{f}_n$, Solid: true function $f_0$ or $f_0'$, Dashes: 95\% $L_\infty$-credible bands.}
\label{fig}
\end{figure}

The first and third rows of Figure~\ref{fig} show that the four methods lead to comparable estimation and uncertainty quantification when estimating $f_0$. However, the deviation between different methods is considerably widened for the estimation of $f_0'$. Mat\'{e}rn and squared exponential kernels constantly give the most accurate point estimation, and their credible bands cover the ground truth with reasonable width, indicating the effectiveness of the nonparametric plug-in procedure using GP priors. There is a tendency for the second-order Sobolev kernel to fail to capture $f_0'$ around the left endpoint, particularly in Simulation 1. The performance of the B-spline method continues to exhibit sensitivity to the choice of $N$, selected by leave-one-out cross validation. In Simulation 1, the selected $N$ is 1, and the B-spline prior yields comparable credible bands of $f_0'$ than GP priors with slightly altered estimation near zero; in Simulation 2 shown in the fourth row, the B-spline method with $N=5$ gives a point estimate that is substantially worse than the other three GP methods, and the associated credible bands are off the chart. We remark that the performance of B-splines might be substantially improved had the number of knots been selected by a different tuning method or with a different simulation setting. While leave-one-out cross validation may be appropriate for estimating $f_0$, as observed in~\cite{yoo2016supremum}, our results suggest that adjustments or alternative strategies seem to be needed when the objective is to make inference on $f_0'$. For GP priors, our numerical results suggest choosing $\lambda$ using empirical Bayes seems to be a reasonable strategy for both $f_0$ and $f_0'$, which is in line with the nonparametric plug-in property of GP priors.

\section{Proofs}\label{sec:proof}
This section contains all the proofs in the paper.
\subsection{Proofs of theorems in Section~\ref{sec:main.results} to Section~\ref{sec:MMLE}}

\begin{proof}[Proof of Theorem~\ref{thm:contraction}]
We first prove that if $\epsilon_n$ is a convergence rate, i.e., $\|\hat{f}_n - f_0 \|_p = O_P(\epsilon_n)$, it is also a contraction rate of $\Pi_n(\cdot\mid\data)$ at $f_0$. Define a centered GP posterior $Z$ as
\begin{equation}\label{eq:centered.GP}
Z:=(f-\hat{f}_n)|\data\sim {\rm GP}(0, \hat{V}_n),
\end{equation}
with psuedo-metric $\rho(\bx,\bx')=\sqrt{{\rm Var}(Z(\bx)-Z(\bx'))}$ for $\bx,\bx'\in\mX$. By Borell-TIS inequality (cf. Proposition A.2.1 in \cite{van1996weak}), we have for any $\epsilon>0$,
\begin{equation}
\Pi_n\left(\left| \|f-\hat{f}_n\|_\infty-\EE\|f-\hat{f}_n\|_\infty\right|>\epsilon\,\big|\,\data\right)\leq 2\exp\left(-\epsilon^2/2\|\hat{V}_n\|_\infty\right).
\end{equation}
Note that
\begin{equation} \Pi_n\left(\Big|\|f-\hat{f}_n\|_\infty-\EE\|f-\hat{f}_n\|_\infty\Big|>\epsilon\,\big|\,\data\right)\geq \Pi_n\left(\|f-\hat{f}_n\|_\infty-\EE\|f-\hat{f}_n\|_\infty>\epsilon\,\big|\,\data\right).
\end{equation}
By Theorem~\ref{thm:equivalent.sigma}, it holds with $\PP_0^{(n)}$-probability at least $1-n^{-10}$ that
\begin{equation}\label{eq:v.hat}
\|\hat{V}_n\|_\infty\leq \frac{2\sigma^2\tilde{\kappa}^2}{n}.
\end{equation}
Thus, we have
\begin{equation}
\Pi_n\left(\|f-\hat{f}_n\|_\infty-\EE\|f-\hat{f}_n\|_\infty>\epsilon\,\big|\,\data\right)\leq 2\exp\left(-C'n\epsilon^2/\tilde{\kappa}^2\right).
\end{equation}
By Dudley’s entropy integral Theorem and Lemma~\ref{lem:entropy.integral}, there exists $C_1>0$ such that
\begin{equation}
\EE\|f-\hat{f}_n\|_\infty\leq \int_0^{\rho(\mX)} \sqrt{\log N(\epsilon,\mX,\rho)}d\epsilon\leq C_1 \tilde{\kappa}\sqrt{\frac{\log n}{n}},
\end{equation}
where $N(\epsilon,\mX,\rho)$ is the $\epsilon$-covering number, namely, the minimal number of balls of radius $\epsilon$ needed to cover $\mX$ with respect to the metric $\rho$.

We next consider the two cases when $p=\infty$ and $p = 2$ separately. If $p = \infty$, i.e., $\|\hat{f}_n-f_0\|_\infty=O_P(\epsilon_n)$, then with $\PP_0^{(n)}$-probability tending to 1 we have
\begin{equation}\label{eq:supremum.eq1}
\|f-\hat{f}_n\|_\infty\geq \|f-f_0\|_\infty-\|\hat{f}_n-f_0\|_\infty\geq \|f-f_0\|_\infty-C\epsilon_n,
\end{equation}
which implies
\begin{equation}
\Pi_n\left(\|f-f_0\|_\infty>C\epsilon_n+C_1\tilde{\kappa}\sqrt{\frac{\log n}{n}}+\epsilon\,\Big|\,\data\right)\leq 2\exp\left(-C'n\epsilon_n^2/\tilde{\kappa}^2\right).
\end{equation}
Since $\tilde{\kappa}^2=O(n\epsilon_n^2/\log n)$, by letting $\epsilon=M_n\epsilon_n$ for any $M_n\rightarrow\infty$, it holds with $\PP_0^{(n)}$-probability tending to 1 that
\begin{equation}
\Pi_n\left(\|f-f_0\|_\infty>M_n\epsilon_n\,\big|\,\data\right)\rightarrow 0.
\end{equation}
Let $A_n$ denote the event that the preceding display holds, which satisfies $\PP_0^{(n)}(A_n^c)\rightarrow0$. Hence,
\begin{equation}
\PP_0^{(n)}\Pi_n\left(\|f-f_0\|_\infty>M_n\epsilon_n\,\big|\,\data\right)\leq \PP_0^{(n)}\Pi_n\left(\|f-f_0\|_\infty>M_n\epsilon_n\,\big|\,\data\right)\mathbbm{1}(A_n)+\PP_0^{(n)}(A_n^c)
\rightarrow 0.
\end{equation}

Now we consider the case of $p=2$. Since $\|\hat{f}_n-f_0\|_2=O_P(\epsilon_n)$, with $\PP_0^{(n)}$-probability tending to 1 we have
\begin{equation}
\|f-\hat{f}_n\|_\infty\geq \|f-\hat{f}_n\|_2\geq \|f-f_0\|_2-\|\hat{f}_n-f_0\|_2\geq \|f-f_0\|_2-C\epsilon_n.
\end{equation}
Comparing the preceding display with \eqref{eq:supremum.eq1} and following the same arguments, we have with $\PP_0^{(n)}$-probability tending to 1 that
\begin{equation}
\Pi_n\left(\|f-f_0\|_2 > M_n\epsilon_n\,\big|\,\data\right)\leq 2\exp\left(-C'n\epsilon_n^2/\tilde{\kappa}^2\right).
\end{equation}
Then a similar argument as in the case of $p = \infty$ yields that $\epsilon_n$ is also a posterior contraction rate. 

We then show the other direction in the theorem. Suppose we have
\begin{equation}
\Pi_n\left(f: \|f-f_0\|_p\geq M_n \epsilon_n \,\big|\, \data\right)\rightarrow 0
\end{equation}
in $\PP_0^{(n)}$-probability for $p=2,\infty$. Consider the ball $\mathcal{B}_{n,p}(\epsilon_n)=\{f: \|f-f_0\|_p<\epsilon_n\}$ of the metric space $(\Ltwo, \|\cdot\|_2)$ if $p = 2$ and $(L^\infty(\mX), \|\cdot\|_\infty)$ if $p = \infty$. Note that $\|\cdot\|_p$ is convex and uniformly bounded on the ball $\mathcal{B}_{n,p}(\epsilon_n)$. By Theorem 8.8 in \cite{ghosal2017fundamentals}, we have
\begin{equation}
\|\hat{f}_n-f_0\|_p\leq M_n\epsilon_n+\|d\|_\infty^{1/2}\Pi_n(\|f-f_0\|_p>M_n\epsilon_n\mid\data),
\end{equation}
conditional on $\mathcal{B}_{n,p}(\epsilon_n)$, where $\|d\|_\infty=\sup_{\bx,\bx'\in \mathcal{B}_{n,p}(\epsilon_n)}d(\bx,\bx')$. Since $\Pi_n(\|f-f_0\|_p>M_n\epsilon_n\mid\data)$ is exponentially small in $\PP_0^{(n)}$-probability, it follows that $\|\hat{f}_n-f_0\|_p=O_P(\epsilon_n)$.
\end{proof}

\begin{proof}[Proof of Theorem~\ref{thm:matern.contraction}]
In view of Lemma~\ref{lem:minimax.matern}, by choosing $\lambda\asymp ({\log n}/{n})^{\frac{2\alpha}{2\alpha+1}}$, with $\PP_0^{(n)}$-probability at least $1-n^{-10}$ we have
\begin{equation}
\|\hat{f}_n-f_0\|_2\lesssim \left(\frac{\log n}{n}\right)^{\frac{\alpha}{2\alpha+1}}.
\end{equation}
The corresponding $\tilde{\kappa}_\alpha^2\asymp\lambda^{-1/2\alpha}\asymp ({n}/{\log n})^{\frac{1}{2\alpha+1}}$ satisfies that $\tilde{\kappa}_\alpha^2=o(\sqrt{n/\log n})$ and $\tilde{\kappa}_\alpha^2=O(n\epsilon_n^2/\log n)$ for any $\alpha>1/2$. Also we have $\hat{\kappa}_{\alpha,01}^2\lesssim \lambda^{-1/\alpha}\asymp (n/\log n)^{\frac{2}{2\alpha+1}}=o(n)$ for any $\alpha>1/2$. The proof is completed following Theorem~\ref{thm:contraction}.
\end{proof}

\begin{proof}[Proof of Theorem~\ref{thm:exponential.contraction}]
In view of Lemma~\ref{lem:minimax.exp}, by choosing $\lambda\asymp 1/n$, with $\PP_0^{(n)}$-probability at least $1-n^{-10}$ we have
\begin{equation}
\|\hat{f}_n-f_0\|_2\lesssim \frac{\log n}{\sqrt{n}}.
\end{equation}
We take $\epsilon_n = {\log n}/{\sqrt{n}}$. It then suffices to verify the conditions in Theorem~\ref{thm:contraction}. 

We write $\tilde{\kappa}^2_\gamma$ and $\hat{\kappa}^2_{\gamma,01}$ as the specialized counterparts of $\tilde{\kappa}^2$ and $\hat{\kappa}_{01}^2$ for kernels with exponentially decaying eigenvalues. According to \eqref{eq:exp.kappa}, $\tilde{\kappa}_\gamma^2\asymp \log n$ satisfies that $\tilde{\kappa}_\gamma^2=o(\sqrt{n/\log n})$ and $\tilde{\kappa}_\gamma^2=O(n\epsilon_n^2/\log n)=O(\log n)$ for any $\gamma>0$. From \eqref{eq:tilde01} we also note that
\begin{equation}
\hat{\kappa}_{\gamma,01}^2\lesssim\sum_{i=1}^{\infty}\frac{i\mu_i}{n^{-1}+\mu_i}\lesssim\sum_{i=1}^{\infty}\frac{i}{1+n^{-1} e^{2\gamma i}}.
\end{equation}
There exists $c_1>0$ such that $\sum_{i=1}^{\infty}\frac{1}{1+n^{-1} e^{2\gamma i}}\leq c_1\log n$. Let $N=\lfloor c_1\log n\rfloor$, then we have
\begin{equation}
\sum_{i=1}^{N}\frac{i}{1+n^{-1} e^{2\gamma i}}\leq \sum_{i=1}^{N}\frac{c_1\log n}{1+n^{-1} e^{2\gamma i}}\leq (c_1\log n)^2.
\end{equation}
When $i\geq N+1$, there exists a constant $c_2>0$ such that $i^{-1}e^{2\gamma i}\gtrsim n+e^{2c_2\gamma i}$. Hence,
\begin{equation}
\sum_{i=N+1}^{\infty}\frac{i}{1+n^{-1} e^{2\gamma i}}=\sum_{i=N+1}^{\infty}\frac{1}{i^{-1}+n^{-1} i^{-1}e^{2\gamma i}}\lesssim \sum_{i=N+1}^{\infty}\frac{1}{1+n^{-1} e^{2c_2\gamma i}}\lesssim \log n.
\end{equation}
Therefore, we have
\begin{equation}
\hat{\kappa}_{\gamma,01}^2\lesssim (\log n)^{2}=o(n),
\end{equation}
and thus the conditions in Theorem~\ref{thm:contraction} are satisfied. This completes the proof. 
\end{proof}

\begin{proof}[Proof of Theorem~\ref{thm:deriv.contraction.l2}]
The proof follows similar arguments as in  Theorem~\ref{thm:contraction} but invokes a number of new error bounds established for derivatives. To delineate the differences, we proceed to show that $\epsilon_n$ is a posterior contraction rate if it is a convergence rate when $p = \infty$. 

Define a centered GP posterior $\tilde{Z}^k$ as
\begin{equation}\label{eq:centered.GP.tilde}
\tilde{Z}^k:=(f^{(k)}-\hat{f}^{(k)}_n)|\data\sim {\rm GP}(0, \tilde{V}_n^k),
\end{equation}
with psuedo-metric $\tilde{\rho}(x,x')=\sqrt{{\rm Var}(\tilde{Z}^k(x)-\tilde{Z}^k(x'))}$ for $x,x'\in[0,1]$. According to Borell-TIS inequality, we have
\begin{equation}
\Pi_{n,k}\left(\left|\|f^{(k)}-\hat{f}^{(k)}_n\|_\infty-\EE\|f^{(k)}-\hat{f}^{(k)}_n\|_\infty\right|>\epsilon\,\big|\,\data\right)\leq 2\exp\left(-\epsilon^2/2\|\tilde{V}^{k}_n\|_\infty\right).
\end{equation}
By Theorem~\ref{thm:equivalent.sigma.deriv}, it holds with $\PP_0^{(n)}$-probability at least $1-n^{-10}$ that
\begin{equation}
\|\tilde{V}^{k}_n\|_\infty \leq \frac{2\sigma^2\tilde{\kappa}_{kk}^2}{n}.
\end{equation}
Thus, we have
\begin{equation}
\Pi_{n,k}\left(\|f^{(k)}-\hat{f}^{(k)}_n\|_\infty-\EE\|f^{(k)}-\hat{f}^{(k)}_n\|_\infty>\epsilon\,\big|\,\data\right)\leq 2\exp\left(-C_1n\epsilon^2/\tilde{\kappa}_{kk}^2\right).
\end{equation}
By Dudley’s entropy integral Theorem and Lemma~\ref{lem:entropy.integral.deriv}, there exists $C_2>0$ such that
\begin{equation}\label{eq:entropy.integral.deriv}
\EE\|f^{(k)}-\hat{f}^{(k)}_n\|_\infty\leq \int_0^{\tilde{\rho}([0,1])} \sqrt{\log N(\epsilon,[0,1],\tilde{\rho})}d\epsilon\leq C_2 \tilde{\kappa}_{kk}\sqrt{\frac{\log n}{n}}.
\end{equation}
Since $\|\hat{f}_n^{(k)}-f_0^{(k)}\|_\infty=O_P(\epsilon_n)$, with $\PP_0^{(n)}$-probability tending to 1 we have
\begin{equation}
\Pi_{n,k}\left(\|f^{(k)}-f_0^{(k)}\|_\infty>C\epsilon_n+C_1\tilde{\kappa}_{kk}\sqrt{\frac{\log n}{n}}+\epsilon\,\big|\,\data\right)\leq 2\exp\left(-C'n\epsilon^2/\tilde{\kappa}^2\right).
\end{equation}
Since $\tilde{\kappa}_{kk}^2=O(n\epsilon_n^2/\log n)$, by letting $\epsilon=M_n\epsilon_n$ for any $M_n\rightarrow\infty$, it holds with $\PP_0^{(n)}$-probability tending to 1 that
\begin{equation}
\Pi_{n,k}\left(\|f^{(k)}-f_0^{(k)}\|_\infty>M_n\epsilon_n\,\big|\,\data\right)=o(1).
\end{equation}
Therefore, $\epsilon_n$ is a contraction rate under the $L_\infty$ norm. 

The case when $p = 2$ and the converse statements for $p = 2$ and $p = \infty$ follow from the same arguments used in proving the counterpart results in Theorem~\ref{thm:contraction}. This completes the proof. 
\end{proof}

\begin{proof}[Proof of Theorem~\ref{thm:contraction.deriv}]
Theorem~\ref{thm:minimax.diff.matern} provides the convergence rate of $\hat{f}_n^{(k)}$ under the $L_2$ norm to be $\epsilon_n=(\log n/n)^{\frac{\alpha-k}{2\alpha+1}}$. For it to become a contraction rate of the posterior distribution $\Pi_{n,k}(\cdot \mid \data)$, by our equivalence theory, we only need to verify the conditions in Theorem~\ref{thm:deriv.contraction.l2}.

Note that the Fourier basis $\{\psi_i\}_{i=1}^\infty$ satisfies Condition (B) with $L_{k,\phi}=\sqrt{2}(2\pi)^{k+1}$. According to Lemma~\ref{lem:differentiability.matern}, we have $\tilde{\kappa}_{\alpha,mm}^2\asymp \lambda^{-\frac{2m+1}{2\alpha}}$ for any $m<\alpha-1/2$ and $m\in\mathbb{N}$. Since $\lambda\asymp (\log n/n)^{\frac{2\alpha}{2\alpha+1}}$, we have $\tilde{\kappa}_\alpha^2\asymp\lambda^{-1/2\alpha}\asymp ({n}/{\log n})^{\frac{1}{2\alpha+1}}=o(\sqrt{n/\log n})$. It also follows that
\begin{equation}
\tilde{\kappa}_{\alpha,kk}\hat{\kappa}_{\alpha,k+1,k+1}\lesssim \lambda^{-\frac{2k+1}{4\alpha}}\lambda^{-\frac{2k+3}{4\alpha}}\asymp\left(\frac{n}{\log n}\right)^{\frac{2(k+1)}{2\alpha+1}}=o(n),
\end{equation}
for any $k<\alpha-3/2$ and $k\in\mathbb{N}$. Moreover, we can see that
\begin{equation}
\tilde{\kappa}_{\alpha,kk}^2\asymp \lambda^{-\frac{2k+1}{2\alpha}}\asymp n^{\frac{2k+1}{2\alpha+1}} (\log n) ^{-\frac{2k+1}{2\alpha+1}}=O(n\epsilon_n^2/\log n).
\end{equation}
This completes the proof. 
\end{proof}

\begin{proof}[Proof of Theorem~\ref{thm:eb.contraction}]
Let $u_1\geq u_2\geq \ldots\geq u_n$ denote the eigenvalues of $K(X,X))$. Since $K(X,X)$ is non-negative definite, we have $u_i\geq 0$ for $1\leq i\leq n$. Note that $\sup_{\bx\in\mX}K(\bx,\bx)<\infty$ as $K$ is a continuous bivariate function on a compact support $\mX\times \mX$. Then we have $\sum_{i=1}^{n}u_i=\tr(K(X,X))\lesssim n$. 
Let $\bm f=(f(X_1),\ldots,f(X_n))^T$. The MMLE $\hat{\sigma}_n^2$ is a quadratic form in $Y$. Hence, in view of the well known formula for the expectation of quadratic forms (cf. Theorem 11.19 in \cite{schott2016matrix}), we obtain
\begin{equation}
\EE(\hat{\sigma}^2_n)=\lambda \sigma_0^2\tr([K(X,X)+n\lambda \bI_n]^{-1})+\lambda\bm f^T[K(X,X)+n\lambda \bI_n]^{-1}\bm f.
\end{equation}
Therefore,
\begin{align}
|\EE(\hat{\sigma}^2_n)-\sigma_0^2| &\leq \left|\lambda \sigma_0^2\tr([K(X,X)+n\lambda \bI_n]^{-1})-\sigma_0^2\right|+\lambda\bm f^T[K(X,X)+n\lambda \bI_n]^{-1}\bm f\\
&\leq \left|n^{-1} \sigma_0^2\tr([(n\lambda)^{-1}K(X,X)+\bI_n]^{-1})-\sigma_0^2\right|+\lambda\bm f^T[K(X,X)+n\lambda \bI_n]^{-1}\bm f.
\end{align}
It follows that the first term is bounded by
\begin{equation}\label{eq:variance.1}
\begin{aligned}
&\left|n^{-1} \sigma_0^2\tr([(n\lambda)^{-1}K(X,X)+\bI_n]^{-1})-\sigma_0^2\right|\\
&\qquad\qquad= n^{-1}\sigma_0^2 \tr(\bI_n-[(n\lambda)^{-1}K(X,X)+\bI_n]^{-1})
=n^{-1}\sigma_0^2\sum_{i=1}^n\left(1-\frac{1}{u_i/n\lambda+1}\right)\\
&\qquad\qquad= n^{-1}\sigma_0^2\sum_{i=1}^{n}\frac{u_i}{n\lambda+u_i}
\leq n^{-1}\sigma_0^2\sum_{i=1}^{n} \frac{u_i}{n\lambda}
\lesssim (n\lambda)^{-1}.
\end{aligned}
\end{equation}
Let $\lambda_{\max}(A)$ denote the largest eigenvalue of a matrix $A$. We have $\lambda_{\max}([K(X,X)+n\lambda \bI_n]^{-1})=(u_n+n\lambda)^{-1}\leq (n\lambda)^{-1}$. Hence,
\begin{equation}\label{eq:variance.2}
\lambda\bm f^T[K(X,X)+n\lambda \bI_n]^{-1}\bm f \leq \lambda\cdot \lambda_{\max}([K(X,X)+n\lambda \bI_n]^{-1})\|\bm f\|_2^2
\leq \|\bm f\|_\infty^2n^{-1}.
\end{equation}
Combining \eqref{eq:variance.1} and \eqref{eq:variance.2} gives
\begin{equation}
|\EE(\hat{\sigma}^2_n-\sigma_0^2)|\lesssim (n\lambda)^{-1}.
\end{equation}

We now bound the variance of $\hat{\sigma}^2_n$. Using the variance formula for quadratic forms (cf. Theorem 11.23 in \cite{schott2016matrix}), we have
\begin{align}
\Var(\hat{\sigma}^2_n)&=2\lambda^2\sigma_0^4\tr([K(X,X)+n\lambda \bI_n]^{-2})+4\lambda^2 \sigma_0^2\bm f^T[K(X,X)+n\lambda \bI_n]^{-2}\bm f\\
&\leq 2\lambda^2\sigma_0^4\cdot n(n\lambda)^{-2}+4\lambda^2 \sigma_0^2\cdot \lambda_{\max}([K(X,X)+n\lambda \bI_n]^{-2})\|\bm f\|_2^2\\
&\leq 2\sigma_0^4 n^{-1}+4\sigma_0^2\|\bm f\|_2^2 n^{-2}.
\end{align}
Therefore,
\begin{equation}
\EE(\hat{\sigma}^2_n-\sigma_0^2)^2=\Var(\hat{\sigma}^2_n-\sigma_0^2)+[\EE(\hat{\sigma}^2_n-\sigma_0^2)]^2\lesssim n^{-1}+(n\lambda)^{-2}=o(1).
\end{equation}
It follows that $\hat{\sigma}^2_n$ converges to $\sigma_0^2$ in $\PP_0^{(n)}$-probability by applying Chebyshev's inequality.

Now we prove Theorem~\ref{thm:contraction} under the empirical Bayes scheme as an example. The results for the differential operator (i.e., Theorem~\ref{thm:deriv.contraction.l2}) and minimax rates in specific examples (Theorem~\ref{thm:matern.contraction}, Theorem~\ref{thm:exponential.contraction} and Theorem~\ref{thm:contraction.deriv}) follow similar arguments. Consider $\mathcal{B}_n$ to be a shrinking neighborhood of $\sigma_0^2$ such that $\PP_0^{(n)}(\hat{\sigma}^2_n\in\mathcal{B}_n)\rightarrow 1$. Conditional on $\mathcal{B}_n$, \eqref{eq:v.hat} becomes 
\begin{equation}
\|\hat{V}_n\|_\infty\leq \frac{2(\sigma_0^2+o(1))\tilde{\kappa}^2}{n}.
\end{equation}
Then, all the established inequalities in the proof of Theorem~\ref{thm:contraction} hold uniformly over $\sigma^2\in \mathcal{B}_n$. In particular, given the convergence rate of $\hat{f}_n$, it follows that
\begin{equation}
\sup_{\sigma^2\in\mathcal{B}_n}\Pi_n\left(f: \|f-f_0\|_p\geq M_n \epsilon_n \big|\, \data,\sigma^2\right)\rightarrow 0
\end{equation}
in $\PP_0^{(n)}$-probability for $p=2,\infty$, which directly implies that
\begin{equation}
\Pi_{n,\text{EB}}\left(f: \|f-f_0\|_p\geq M_n \epsilon_n \big|\, \data\right)\rightarrow 0
\end{equation}
in $\PP_0^{(n)}$-probability for $p=2,\infty$. That posterior contraction rates of $\Pi_{n,\text{EB}}(\cdot\mid\data)$ imply convergence rates of $\hat{f}_n$ follows the same argument as in Theorem~\ref{thm:contraction}. This completes the proof. 
\end{proof}

\subsection{Proofs in Section~\ref{sec:non-asymptotic}}

\begin{proof}[Proof of Corollary~\ref{cor:h.tilde.noiseless}]
We first consider the inequality with respect to the $\|\cdot\|_{\tilde{\bbH}}$ norm. If $\tilde{\kappa}^2=o(\sqrt{n/\log n})$, then for sufficiently large $n$ we have $C(n,\tilde{\kappa})\leq 1/2$. The noise-free version of Theorem~\ref{thm:equivalent.bound} yields that with $\PP_0^{(n)}$-probability at least $1-n^{-10}$,
\begin{equation}
\|\hat{f}_n-f_\lambda\|_{\tilde{\bbH}}\leq \tilde{\kappa}^{-1}\|f_\lambda-f_0\|_\infty.
\end{equation}
Then we have
\begin{align}
\|\hat{f}_n-f_0\|_{\tilde{\bbH}}&\leq\|\hat{f}_n-f_\lambda\|_{\tilde{\bbH}}+\|f_\lambda-f_0\|_{\tilde{\bbH}}\\
&\leq \tilde{\kappa}^{-1}\|f_\lambda-f_0\|_\infty+\|f_\lambda-f_0\|_{\tilde{\bbH}}\\
&\leq 2\|f_\lambda-f_0\|_{\tilde{\bbH}},
\end{align}
where the last inequality follows from \eqref{eq:sup.RKHS.norm}.

We then prove the error bound under the $L_\infty$ norm. Applying \eqref{eq:sup.RKHS.norm} to Theorem~\ref{thm:equivalent.bound} with $\sigma=0$, it holds with $\PP_0^{(n)}$-probability at least $1-n^{-10}$ that
\begin{equation}
\|\hat{f}_n-f_\lambda\|_\infty\leq \frac{C(n,\tilde{\kappa})}{1-C(n,\tilde{\kappa})}\|f_\lambda-f_0\|_\infty.
\end{equation}
With the same choice of $\lambda$ such that $\tilde{\kappa}^2=o(\sqrt{\log n/n})$ and $C(n,\tilde{\kappa})\leq 1/2$, the preceding display becomes
\begin{equation}
\|\hat{f}_n-f_\lambda\|_\infty\leq \|f_\lambda-f_0\|_\infty.
\end{equation}
The proof is completed by the triangle inequality.
\end{proof}

\begin{proof}[Proof of Lemma~\ref{lem:minimax.exp}]
Considering the equivalent kernel $\tilde{K}_\gamma$, from \eqref{eq:def.kappa} we have
\begin{equation}\label{eq:exp.kappa}
\begin{split}
\tilde{\kappa}_\gamma^2&=\sup_{x\in[0,1]}\tilde{K}_\gamma(x,x)\lesssim\sum_{i=1}^{\infty}\frac{\mu_i}{\lambda+\mu_i}\lesssim\sum_{i=1}^{\infty}\frac{1}{1+\lambda e^{2\gamma i}}\\
&\leq \int_0^{\infty}\frac{dx}{1+\lambda e^{2\gamma x}}=\frac{\log(1+\lambda^{-1})}{2\gamma}\lesssim -\log\lambda.
\end{split}
\end{equation}
Substituting $\tilde{\kappa}_\gamma\lesssim \sqrt{-\log\lambda}$ into Theorem 5 in \cite{liu2020non}, we obtain that with $\PP_0^{(n)}$-probability at least $1-n^{-10}$ it holds
\begin{equation}\label{eq:exponential.eq1}
\|\hat{f}_n-f_0\|_2\lesssim 2\|f_\lambda-f_0\|_\infty+ \frac{8\sigma \sqrt{20\log n}\sqrt{-\log\lambda}}{\sqrt{n}},
\end{equation}
where $\lambda$ is chosen to satisfy that $\tilde{\kappa}_\gamma^2=o(\sqrt{n/\log n})$. 

Let $f_0=\s f_i\phi_i$. Then we have $f_\lambda-f_0=-\s\frac{\lambda}{\lambda+\mu_i}f_i\phi_i$. Hence,
\begin{equation}\label{eq:exponential.eq2}
\|f_\lambda-f_0\|_\infty\lesssim \s \frac{\lambda}{\lambda+\mu_i} |f_i|=\sqrt{\lambda}\s\frac{\sqrt{\lambda\mu_i}}{\lambda+\mu_i}\frac{|f_i|}{\sqrt{\mu_i}}\lesssim \sqrt{\lambda}\s e^{\gamma i}|f_i|\lesssim \sqrt{\lambda}.
\end{equation}
Combining \eqref{eq:exponential.eq2} and \eqref{eq:exponential.eq1} gives
\begin{equation}
\label{eq:dummay1.11.20}
\|\hat{f}_n-f_0\|_2\lesssim 2\sqrt{\lambda}+ \frac{8 \sigma \sqrt{20\log n}\sqrt{-\log\lambda}}{\sqrt{n}}.
\end{equation}
The upper bound is minimized at $\lambda\asymp 1/n$, which satisfies $\tilde{\kappa}_\gamma^2\lesssim \log n=o(\sqrt{n/\log n})$ for any $\gamma>0$. The proof is completed by substituting this $\lambda$ into~\eqref{eq:dummay1.11.20}. 
\end{proof}

\begin{proof}[Proof of Lemma~\ref{lem:differentiability.matern}]
Recall the definition of $\tilde{\kappa}_{kk}^2$ in \eqref{eq:high.order.kappa}. It follows that for any $m\in\mathbb{N}_0$,
\begin{align}
\tilde{\kappa}_{\alpha,mm}^2&=\sup_{x\in[0,1]}\s\frac{\mu_i}{\lambda+\mu_i}\psi^{(m)}_i(x)^2\\
&\lesssim \sum_{i=1}^{\infty}\frac{i^{2m}}{1+\lambda i^{2\alpha}}\leq
\int_0^{\infty}\frac{(x+1)^{2m}dx}{1+\lambda x^{2\alpha}}\asymp  \lambda^{-\frac{2m+1}{2\alpha}},
\end{align}
where the last step holds for $\alpha>m+\frac{1}{2}$. On the other hand, we have
\begin{align}
\tilde{\kappa}_{\alpha,mm}^2&\gtrsim\s\left[\frac{(2i)^{2m}}{1+\lambda (2i)^{2\alpha}}\cos(2\pi ix)^2+\frac{(2i+1)^{2m}}{1+\lambda (2i+1)^{2\alpha}}\sin(2\pi ix)^2\right]\\
&\geq \s \min\left\{\frac{(2i)^{2m}}{1+\lambda (2i)^{2\alpha}},\frac{(2i+1)^{2m}}{1+\lambda (2i+1)^{2\alpha}}\right\}\\
&\geq \frac{1}{2} \sum_{i=1}^{\infty}\frac{i^{2m}}{1+\lambda i^{2\alpha}}\asymp  \lambda^{-\frac{2m+1}{2\alpha}},
\end{align}
where we also need $\alpha>m+\frac{1}{2}$. The differentiability of $\tilde{K}_\alpha$ directly follows from the boundedness of $\tilde{\kappa}^2_{\alpha,mm}$ for any fixed $\lambda$.

The rates for $\hat{\kappa}_{\alpha, 01}^2$ and $\hat{\kappa}_{\alpha, k+1, k+1}^2$, defined in \eqref{eq:tilde01} and \eqref{eq:kappa.k+1}, can be  obtained by direct calculation. In particular, we have 
\begin{equation}
\hat{\kappa}_{\alpha, 01}^2=\s\frac{i\mu_i}{\lambda+\mu_i}=\s \frac{i}{1+\lambda i^{2\alpha}} \lesssim \lambda^{-\frac{1}{\alpha}}
\end{equation}
and
\begin{equation}
\hat{\kappa}_{\alpha, k+1, k+1}^2=\s\frac{i^{2k+2}\mu_i}{\lambda+\mu_i}=\frac{i^{2k+2}}{1+\lambda i^{2\alpha}} \lesssim \lambda^{-\frac{2k+3}{2\alpha}},
\end{equation}
which hold when $\alpha>1$ and $\alpha>k+\frac{3}{2}$, respectively.
\end{proof}

\begin{proof}[Proof of Lemma~\ref{lem:matern.RKHS.derivative}]
In view of Corollary 4.36 in \cite{steinwart2008support}, $\tilde{K}_\alpha\in C^{2m}([0,1]\times[0,1])$ implies that $f\in C^m[0,1]$ for any $f \in \tilde{\bbH}_\alpha$. This is also true for $f \in \bbH_\alpha$ since $\bbH_\alpha$ and $\tilde{\bbH}_\alpha$ contain the same functions.

Now we prove the norm inequality. Let $f=\s f_i\psi_i$ where $\{\psi_i\}_{i=1}^\infty$ is the Fourier basis, then $\|f^{(m)}\|_2^2\asymp \s (f_i i^m)^2$ for any $m\in\mathbb{N}_0$. It is equivalent to showing that
\begin{equation}
\tilde{\kappa}_\alpha^2 \cdot \s f_i^2i^{2m} \leq C\tilde{\kappa}_{\alpha,mm}^2 \cdot \s f_i^2\frac{\lambda+\mu_i}{\mu_i},
\end{equation}
for some $C>0$. Hence, it suffices to show that which is equivalent to showing that for any $i\in\mathbb{N}$,
\begin{equation}
\tilde{\kappa}_\alpha^2 \cdot f_i^2i^{2m}\leq C \tilde{\kappa}_{\alpha,mm}^2 \cdot f_i^2\frac{\lambda+\mu_i}{\mu_i}.
\end{equation}
In view of Lemma~\ref{lem:differentiability.matern}, we have $\tilde{\kappa}_{\alpha,mm}^2 \asymp \lambda^{-\frac{2m+1}{2\alpha}}$, which also leads to $\tilde{\kappa}_\alpha^2 \asymp \lambda^{-\frac{1}{2\alpha}}$ when taking $m = 0$. Since $\mu_i\asymp i^{-2\alpha}$, it is sufficient to show that
\begin{equation}\label{eq:matern.ineq}
\lambda^{\frac{m}{\alpha}} i^{2m}\leq C' (1+\lambda i^{2\alpha}),
\end{equation}
for some constant $C'>0$. The above equation trivially holds for $C' = 1$ if $\lambda^{\frac{m}{\alpha}} i^{2m}\leq 1$. If $\lambda^{\frac{m}{\alpha}} i^{2m}\geq 1$, then $\lambda^{\frac{m}{\alpha}} i^{2m}\leq (\lambda^{\frac{m}{\alpha}} i^{2m})^{\frac{\alpha}{m}}=\lambda i^{2\alpha}$ since $m<\alpha$. Taking $C' = 1$ completes the proof.
\end{proof}

\begin{proof}[Proof of Lemma~\ref{lem:matern.deriv.deterministic}]
Let $f_0=\s f_i\psi_i$. Then, for any $k\in\mathbb{N}_0$,
\begin{equation}
f_\lambda^{(k)}-f_0^{(k)}=-\s\frac{\lambda}{\lambda+\mu_i}f_i\psi_i^{(k)}.
\end{equation}
Hence,
\begin{equation}
\|f_\lambda^{(k)}-f_0^{(k)}\|_\infty\leq \s \frac{\lambda}{\lambda+\mu_i} |f_i|\cdot i^k\lesssim \lambda^{\frac{1}{2}-\frac{k}{2\alpha}}\s \frac{ i^{k-\alpha}\cdot \lambda^{\frac{1}{2}+\frac{k}{2\alpha}}}{i^{-2\alpha}+\lambda} \cdot i^\alpha|f_i|.
\end{equation}
By Young's inequality for products, we have
\begin{equation}
i^{k-\alpha} \cdot \lambda^{\frac{1}{2}+\frac{k}{2\alpha}}\leq \frac{\alpha-k}{2\alpha}\cdot i^{-2\alpha} +\frac{\alpha+k}{2\alpha}\cdot \lambda\lesssim i^{-2\alpha}+\lambda.
\end{equation}
Therefore, $\|f_\lambda^{(k)}-f_0^{(k)}\|_\infty\lesssim \lambda^{\frac{1}{2}-\frac{k}{2\alpha}}\s  i^\alpha|f_i|\lesssim \lambda^{\frac{1}{2}-\frac{k}{2\alpha}}$. This completes the proof. 
\end{proof}

\begin{proof}[Proof of Theorem~\ref{thm:minimax.diff.matern}]
Applying Lemma~\ref{lem:matern.RKHS.derivative} to Theorem~\ref{thm:equivalent.bound} yields with $\PP_0^{(n)}$-probability at least $1-n^{-10}$ that
\begin{align}
\|\hat{f}_n^{(k)}-f_\lambda^{(k)}\|_2& \leq \tilde{\kappa}_\alpha^{-1}\tilde{\kappa}_{\alpha,kk} \|\hat{f}_n-f_\lambda\|_{\tilde{\bbH}_\alpha}\\
&\leq\frac{\tilde{\kappa}_{\alpha,kk}\tilde{\kappa}_\alpha^{-2}C(n,\tilde{\kappa}_\alpha)}{1-C(n,\tilde{\kappa}_\alpha)}\|f_\lambda-f_0\|_\infty+\frac{1}{1-C(n,\tilde{\kappa}_\alpha)}\frac{4\tilde{\kappa}_{\alpha,kk}\sigma\sqrt{20\log n}}{\sqrt{n}}.
\end{align}
Note that $\tilde{\kappa}_\alpha^{-2}C(n,\tilde{\kappa}_\alpha)\asymp \sqrt{\log n/n}$. By choosing $\lambda$ such that $\tilde{\kappa}_\alpha^2=o(\sqrt{n/\log n})$, $\tilde{\kappa}_{\alpha,kk}=o(\sqrt{n/\log n})$ and $C(n,\tilde{\kappa}_\alpha)\leq 1/2$, we arrive at
\begin{equation}
\|\hat{f}_n^{(k)}-f_\lambda^{(k)}\|_2\lesssim  2\|f_\lambda-f_0\|_\infty+\frac{8\tilde{\kappa}_{\alpha,kk}\sigma\sqrt{20\log n}}{\sqrt{n}}.
\end{equation}
According to Lemma~\ref{lem:matern.deriv.deterministic}, we obtain
\begin{align}
\|\hat{f}_n^{(k)}-f_0^{(k)}\|_2 &\lesssim\|f_\lambda^{(k)}-f_0^{(k)}\|_2+ 2\|f_\lambda-f_0\|_\infty+\frac{8\tilde{\kappa}_{\alpha,kk}\sigma\sqrt{20\log n}}{\sqrt{n}}\\
&\leq\|f_\lambda^{(k)}-f_0^{(k)}\|_\infty+ 2\|f_\lambda-f_0\|_\infty+\frac{8\tilde{\kappa}_{\alpha,kk}\sigma\sqrt{20\log n}}{\sqrt{n}}\\
&\lesssim\lambda^{\frac{1}{2}-\frac{k}{2\alpha}}+\lambda^{\frac{1}{2}}+\lambda^{-\frac{2k+1}{4\alpha}}\sqrt{\frac{\log n}{n}}.
\end{align}
The upper bound in the preceding display is minimized when $\lambda\asymp ({\log n}/{n})^{\frac{2\alpha}{2\alpha+1}}$, which satisfies $\tilde{\kappa}_\alpha^2\asymp(n/\log n)^{\frac{1}{2\alpha+1}}=o(\sqrt{n/\log n})$ and $\tilde{\kappa}_{\alpha,kk}\asymp (n/\log n)^{\frac{2k+1}{2(2\alpha+1)}}=o(\sqrt{n/\log n})$. The proof is completed by substituting $\lambda$.
\end{proof}

\begin{proof}[Proof of Theorem~\ref{thm:equivalent.sigma}]
Recall from \eqref{eq:posterior.variance} that we have
\begin{equation}
|\sigma^{-2}n\lambda\hat{V}_n(\bx')| =|\sigma^{-2}n\lambda \hat{V}_n(\bx',\bx')| \leq\|\hat{K}_{\bx'}-K_{\bx'}\|_\infty
\end{equation}
for any $\bx \in \mX$. Since $\hat{K}_{\bx'}$ is the solution to a kernel ridge regression with noiseless observations \eqref{eq:noiseless.krr}, in view of Corollary~\ref{cor:h.tilde.noiseless}, we have with $\PP_0^{(n)}$-probability at least $1 - n^{-10}$ that
\begin{equation}
\|\hat{K}_{\bx'}-K_{\bx'}\|_\infty\leq 2\|L_{\tilde{K}}K_{\bx'}-K_{\bx'}\|_\infty \leq 2\bigg\|\s \frac{\lambda}{\lambda+\mu_i}\mu_i\phi_i(\bx')\phi_i\bigg\|_\infty \leq 2\lambda\tilde{\kappa}^2.
\end{equation}
Therefore, it follows that with $\PP_0^{(n)}$-probability at least $1 - n^{-10}$ we have 
\begin{equation}
\|\hat{V}_n\|_\infty\leq \frac{2\sigma^2\tilde{\kappa}^2}{n}.
\end{equation}
\end{proof}

\begin{proof}[Proof of Theorem~\ref{thm:equivalent.sigma.deriv}]
Let $K_{0k,x'}(\cdot)=\partial_{x'}^kK(\cdot,x')$ and 
$$\hat{K}_{0k,x'}(\cdot)=K(\cdot, X)[K(X, X) + n \lambda \bI_n]^{-1} K_{0k}(X, x').$$ 
It is easy to see that $\hat{K}_{0k,x'}$ is the solution to a noise-free KRR with observations $K_{0k,x'}$. Moreover, we have
\begin{equation}\label{eq:tilde.variance}
\begin{split}
|\sigma^{-2}n\lambda\tilde{V}^{k}_{n}(x')| &= \left|\partial^k_{x}(\hat{K}_{0k,x'}-K_{0k,x'})(x)|_{x=x'}\right|\\
&\leq \|\partial^k_{x}(\hat{K}_{0k,x'}-K_{0k,x'})\|_\infty \leq \tilde{\kappa}_{kk}
\|\hat{K}_{0k,x'}-K_{0k,x'}\|_{\tilde{\bbH}}.
\end{split}
\end{equation}
According to Corollary~\ref{cor:h.tilde.noiseless}, it holds with $\PP_0^{(n)}$-probability at least $1 - n^{-10}$  that
\begin{equation}
\|\hat{K}_{0k,x'}-K_{0k,x'}\|_{\tilde{\bbH}}\leq 2\|L_{\tilde{K}}K_{0k,x'}-K_{0k,x'}\|_{\tilde{\bbH}} \leq 2\lambda \tilde{\kappa}_{kk}.
\end{equation}
\end{proof}

\subsection{Auxiliary results and proofs}

\begin{lem}\label{lem:entropy.integral}
Under the conditions of Theorem~\ref{thm:contraction}, let $\rho$ be the intrinsic psuedo-metric of the centered Gaussian process \eqref{eq:centered.GP}. Then, it holds with $\PP_0^{(n)}$-probability at least $1 - n^{-10}$ that
\begin{align}
\int_0^{\rho(\mX)} \sqrt{\log N(\epsilon,\mX,\rho)}d\epsilon\lesssim \tilde{\kappa}\sqrt{\frac{\log n}{n}}.
\end{align}
\end{lem}

\begin{proof}[Proof of Lemma~\ref{lem:entropy.integral}]

Recall that in \eqref{eq:posterior.variance} the posterior covariance can be expressed as the bias of a noise-free KRR:
\begin{equation}
\sigma^{-2}n\lambda \hat{V}_n(\bx,\bx')=K_{\bx'}(\bx)-\hat{K}_{\bx'}(\bx),
\end{equation}
where $K_{\bx'}(\cdot)=K(\cdot,\bx')$ and $\hat{K}_{\bx'}(\cdot)=K(\cdot, X)[K(X, X) + n \lambda \bI_n]^{-1} K(X, \bx')$. It follows that
\begin{align}
\rho(\bx,\bx')^2&={\rm Var}(Z(\bx),Z(\bx'))
= \hat{V}_n(\bx,\bx)+\hat{V}_n(\bx',\bx')-2\hat{V}_n(\bx,\bx')\\
&= \sigma^2 (n \lambda)^{-1} \{K_{\bx}(\bx) - \hat{K}_{\bx}(\bx)+K_{\bx'}(\bx') - \hat{K}_{\bx'}(\bx')-2K_{\bx'}(\bx) +2 \hat{K}_{\bx'}(\bx)\}\\
&= \sigma^2 (n \lambda)^{-1} \{(K_{\bx}-K_{\bx'})(\bx) - (K_{\bx}-K_{\bx'})(\bx') - (\hat{K}_{\bx}-\hat{K}_{\bx'})(\bx)+ (\hat{K}_{\bx}-\hat{K}_{\bx'})(\bx')\}.
\end{align}
Letting $g=K_{\bx}-K_{\bx'}$ and $\hat{g}_n=\hat{K}_{\bx}-\hat{K}_{\bx'}$, the preceding display implies
\begin{equation}
\sigma^{-2}n\lambda\rho(\bx,\bx')^2=(g-\hat{g}_n)(\bx)-(g-\hat{g}_n)(\bx').
\end{equation}
In view of Corollary~\ref{cor:h.tilde.noiseless}, we have with $\PP_0^{(n)}$-probability at least $1 - n^{-10}$  that
\begin{equation}\label{eq:rho.decomposition}
\sigma^{-2}n\lambda\rho(\bx,\bx')^2 \leq 2\|\hat{g}_n-g\|_\infty\leq 4\|L_{\tilde{K}}g-g\|_\infty.
\end{equation}
Substituting $g=\sum_{i=1}^{\infty}\mu_i(\phi_i(\bx)-\phi_i(\bx'))\phi_i$ yields the following bound 
\begin{equation}\label{eq:rho.euclid}
\begin{split}
\|L_{\tilde{K}}g-g\|_\infty&= \bigg\| \sum_{i=1}^{\infty}\frac{\lambda\mu_i(\phi_i(\bx)-\phi_i(\bx'))\phi_i}{\lambda+\mu_i}\bigg\|_\infty\\
&\leq \lambda L_\phi \|\bx-\bx'\| \s \frac{i\mu_i  \|\phi_i\|_\infty}{\lambda+\mu_i}\\
&\leq C_\phi L_\phi\hat{\kappa}_{01}^2\lambda \|\bx-\bx'\|,
\end{split}
\end{equation}
where the first inequality follows from Condition (A2). Then \eqref{eq:rho.decomposition} and \eqref{eq:rho.euclid} together give
\begin{equation}
\rho(\bx,\bx')\leq C \hat{\kappa}_{01} n^{-1/2}\|\bx-\bx'\|^{1/2}.
\end{equation}
The preceding inequality shows that $\rho$ is bounded above by the Euclidean norm up to a multiplicative constant. By the inequality in \citet[p.~529]{ghosal2017fundamentals}, we have
\begin{equation}
N(\epsilon,\mX,\rho)\leq N\left((C^{-1}\sqrt{n}\epsilon/\hat{\kappa}_{01})^{2},\mX,\|\cdot\|\right)\lesssim \left(\frac{1}{(\sqrt{n}\epsilon/\hat{\kappa}_{01})^{2}}\right)^d.
\end{equation}
Since $\hat{\kappa}_{01}^2=o(n)$, the preceding inequality implies
\begin{equation}
\log N(\epsilon,\mX,\rho)\lesssim 2d\log \left(\frac{1}{\epsilon}\right).
\end{equation}
On the other hand, we have
\begin{align}
\rho(\bx,\bx')^2&\leq 4\sigma^2(n\lambda)^{-1}\|L_{\tilde{K}}g-g\|_\infty\\
&\leq 4\sigma^2(n\lambda)^{-1}\bigg\|\sum_{i=1}^{\infty}\frac{\lambda\mu_i(\phi_i(\bx)-\phi_i(\bx'))\phi_i}{\lambda+\mu_i}\bigg\|_\infty\\
&\leq 4\sigma^2(n\lambda)^{-1} \cdot 2\lambda \tilde{\kappa}^2\\
&\lesssim \tilde{\kappa}^2n^{-1},
\end{align}
which implies $\rho(\mX)\lesssim \tilde{\kappa}/\sqrt{n}$. Therefore,
\begin{align}
\int_0^{\rho(\mX)} \sqrt{\log N(\epsilon,\mX,\rho)}d\epsilon\lesssim \tilde{\kappa}\sqrt{\frac{\log n}{n}}.
\end{align}
\end{proof}

\begin{lem}\label{lem:entropy.integral.deriv}
Under the conditions of Theorem~\ref{thm:deriv.contraction.l2}, let $\tilde{\rho}$ be the intrinsic psuedo-metric of the centered Gaussian process \eqref{eq:centered.GP.tilde}. Then, it holds with $\PP_0^{(n)}$-probability at least $1 - n^{-10}$ that
\begin{equation}
\int_0^{\tilde{\rho}([0,1])} \sqrt{\log N(\epsilon,[0,1],\tilde{\rho})}d\epsilon\lesssim \tilde{\kappa}_{kk}\sqrt{\frac{\log n}{n}}.
\end{equation}
\end{lem}

\begin{proof}[Proof of Lemma~\ref{lem:entropy.integral.deriv}]
Recall that in \eqref{eq:tilde.variance} the posterior covariance can be expressed the bias of a noise-free KRR estimator:
\begin{equation}
\sigma^{-2}n\lambda \tilde{V}^k_n(x,x')=\partial^k_x(K_{0k,x'}-\hat{K}_{0k,x'})(x).
\end{equation}
Hence, we have
\begin{align}
\tilde{\rho}(x,x')^2&={\rm Var}(\tilde{Z}^k(x),\tilde{Z}^k(x'))=\tilde{V}^k_n(x,x)+\tilde{V}^k_n(x',x')-2\tilde{V}^k_n(x,x')\\
&= \sigma^2 (n \lambda)^{-1} \{\partial^k_xK_{0k,x}(x) - \partial^k_x\hat{K}_{0k,x}(x)+\partial^k_{x'}K_{0k,x'}(x') - \partial^k_{x'}\hat{K}_{0k,x'}(x')\\
&\qquad\qquad\quad\ -2\partial^k_xK_{0k,x'}(x) +2\partial^k_x \hat{K}_{0k,x'}(x)\}\\
&= \sigma^2 (n \lambda)^{-1} \{\partial^k_x(K_{0k,x}-K_{0j,x'})(x) - \partial^k_{x'}(K_{0k,x}-K_{0k,x'})(x') \\
&\qquad\qquad\quad\ - \partial^k_x(\hat{K}_{0k,x}-\hat{K}_{0k,x'})(x)+ \partial^k_{x'}(\hat{K}_{0k,x}-\hat{K}_{0k,x'})(x')\}.
\end{align}
Let $g_{0k}=K_{0k,x}-K_{0k,x'}$ and $\hat{g}_{0k}=\hat{K}_{0k,x}-\hat{K}_{0k,x'}$. Then the preceding display implies
\begin{equation}
\sigma^{-2}n\lambda\tilde{\rho}(x,x')^2=\partial^k_x(g_{0k}-\hat{g}_{0k})(x)-\partial^k_{x'}(g_{0k}-\hat{g}_{0k})(x').
\end{equation}
By Corollary~\ref{cor:h.tilde.noiseless}, with $\PP_0^{(n)}$-probability at least $1 - n^{-10}$ it holds that
\begin{equation}\label{eq:rho.decomposition.tilde}
\begin{split}
\sigma^{-2}n\lambda\tilde{\rho}(x,x')^2&\leq 2\|\partial^k(\hat{g}_{0k}-g_{0k})\|_\infty\\
&\leq 2\tilde{\kappa}_{kk}\|\hat{g}_{0k}-g_{0k}\|_{\tilde{\bbH}}\\
&\leq 4\tilde{\kappa}_{kk}\|L_{\tilde{K}}g_{0k}-g_{0k}\|_{\tilde{\bbH}}.
\end{split}
\end{equation}
Substituting $g_{0k}=\sum_{i=1}^{\infty}\mu_i(\phi^{(k)}_i(x)-\phi^{(k)}_i(x'))\phi_i$ yields
\begin{equation}\label{eq:rho.euclid.tilde}
\begin{split}
\|L_{\tilde{K}}g_{0k}-g_{0k}\|_{\tilde{\bbH}}^2&=  \sum_{i=1}^{\infty}\left(\frac{\lambda\mu_i(\phi^{(k)}_i(x)-\phi^{(k)}_i(x'))}{\lambda+\mu_i}\right)^2\bigg/\frac{\mu_i}{\lambda+\mu_i}\\
&\leq \lambda^2 L_{k,\phi}^2 |x-x'|^{2} \s \frac{i^{2k+2}\mu_i}{\lambda+\mu_i}\\
&\leq  L_{k,\phi}^2 \lambda^2 \hat{\kappa}_{k+1,k+1}^2 |x-x'|^{2}.
\end{split}
\end{equation}
It follows from \eqref{eq:rho.decomposition.tilde} and \eqref{eq:rho.euclid.tilde} that
\begin{equation}\label{eq:tilde.rho}
\tilde{\rho}(x,x')\leq C\sqrt{\tilde{\kappa}_{kk}\hat{\kappa}_{k+1,k+1}} n^{-1/2} |x-x'|^{1/2}.
\end{equation}
Again by the inequality in \citet[p.~529]{ghosal2017fundamentals}, we have
\begin{equation}
N(\epsilon,[0,1],\tilde{\rho})\leq N\left(\left(C^{-1}\epsilon\sqrt{n/\tilde{\kappa}_{kk}\hat{\kappa}_{k+1,k+1}}\right)^{2},[0,1],|\cdot|\right)\lesssim \frac{1}{(\epsilon\sqrt{n/\tilde{\kappa}_{kk}\hat{\kappa}_{k+1,k+1}})^{2}}.
\end{equation}
Since $\tilde{\kappa}_{kk}\hat{\kappa}_{k+1,k+1}=o(n)$, we obtain $\log N(\epsilon,[0,1],\tilde{\rho})\lesssim \log(1/\epsilon)$. On the other hand, note that
\begin{align}
\|L_{\tilde{K}}g-g\|_{\tilde{\bbH}}^2&=  \sum_{i=1}^{\infty}\left(\frac{\lambda\mu_i(\phi_i^{(k)}(x)-\phi_i^{(k)}(x'))}{\lambda+\mu_i}\right)^2\bigg/\frac{\mu_i}{\lambda+\mu_i}\\
&\leq \sum_{i=1}^{\infty}\frac{2\lambda^2\mu_i(\phi_i^{(k)}(x)^2+\phi_i^{(k)}(x')^2)}{\lambda+\mu_i}\\
&\leq 4\lambda^2\tilde{\kappa}_{kk}^2.
\end{align}
The preceding inequality combined with \eqref{eq:rho.decomposition.tilde} gives that
\begin{equation}
\tilde{\rho}(x,x')^2\leq 4\sigma^2(n\lambda)^{-1}\tilde{\kappa}_{kk}\|L_{\tilde{K}}g-g\|_{\tilde{\bbH}}\leq 8\sigma^2 \tilde{\kappa}_{kk}^2 /n,
\end{equation}
which implies $\tilde{\rho}([0,1])\lesssim \tilde{\kappa}_{kk}/\sqrt{n}$. Therefore,
\begin{equation}
\int_0^{\tilde{\rho}([0,1])} \sqrt{\log N(\epsilon,[0,1],\tilde{\rho})}d\epsilon \lesssim \int_0^{\tilde{\kappa}_{kk}/\sqrt{n}} \sqrt{\log(1/\epsilon)}d\epsilon \lesssim \tilde{\kappa}_{kk}\sqrt{\frac{\log n}{n}}.
\end{equation}
\end{proof}

\section*{Acknowledgements}
We would like to thank William Yoo for providing R code to implement the B-spline prior in the simulation section. 

\bibliographystyle{apalike}
\bibliography{EquivalencePlugin}

\end{document}